\documentclass[reqno]{amsart}

\usepackage{amsthm,amsmath,amsfonts,amssymb,epsfig,graphicx}

\newtheorem{definition}{Definition}
\newtheorem{theorem}[definition]{Theorem}
\newtheorem{proposition}[definition]{Proposition}
\newtheorem{lemma}[definition]{Lemma}

\newtheorem{assumption}[definition]{Assumption}
\newtheorem{remark}[definition]{Remark}

\begin{document}

\title[Analysis and numerics for nonlinear PDAEs arising in technical textile industry]
{On a nonlinear partial differential algebraic system arising in technical textile industry:\\ Analysis and numerics}

\author[M. Grothaus]{Martin Grothaus$^{1,\star}$}
\author[N. Marheineke]{Nicole Marheineke$^2$}

\date{\today\\
$^\star$ \emph{Corresponding author}, email: grothaus@mathematik.uni-kl.de\\
$^1$ TU Kaiserslautern, Fachbereich Mathematik, D-67653 Kaiserslautern, Germany\\
$^2$ FAU Erlangen-N\"urnberg, Lehrstuhl Angewandte Mathematik I, Cauerstr.\ 11, D-91058 Erlangen, Germany}

\begin{abstract}
In this paper we explore a numerical scheme for a nonlinear fourth order system of partial differential algebraic equations that describes the dynamics of slender inextensible elastica as they arise in the technical textile industry. Applying a semi-discretization in time, the resulting sequence of nonlinear elliptic systems with the algebraic constraint for the local length preservation is reformulated as constrained optimization problems in a Hilbert space setting that admit a solution at each time level. Stability and convergence of the scheme are proved. The numerical realization is based on a finite element discretization in space. The simulation results confirm the analytically predicted properties of the scheme.
\end{abstract}

\maketitle

\quad\\
\textsc{AMS-Classification} 35J74; 58J05; 65K10; 65M12; 65M20; 65M60; 74K10\\
\textsc{Keywords} numerical scheme; stability; convergence; semi-discretization; constrained optimization; finite elements; elastic fiber dynamics

%%%%%%%%%%%%%%%%%%%%%%%%
\section{Introduction}

The numerical simulation and optimization of the dynamics of thin long elastic fibers are of great importance in the technical textile industry  (e.g.\ in production processes of yarns or non-woven materials \cite{pearson:b:1985, klar:p:2009}), but the application ranges much further and comprises also, among others, biomolecular science (DNA, bacterial fibers \cite{mesirov:b:1996}) and computer graphics (hair modeling \cite{bertails:p:2006}).
In the slender-body theory \cite{antman:b:2006} a fiber can be asymptotically described by an arc-length parameterized, time-dependent curve $\mathbf{r}$ representing its center-line. Then, its dynamics can be modeled by a system of nonlinear partial differential equations  \cite{marheineke:p:2006}
\begin{align}\label{eq:intro}
\omega \,\partial_{tt} \mathbf{r} &= \partial_s(\lambda \partial_s \mathbf{r})-b \, \partial_{ssss}\mathbf{r} + \mathbf{f}, \qquad \qquad |\partial_s \mathbf{r}(s,t)|^2=1. 
\end{align}
The arc-length constraint enforces inextensibility and turns the inner traction $\lambda$ to an unknown, i.e.\ Lagrange multiplier. The system for $(\mathbf{r}, \lambda) $ has a wave-like character due to inertia (line weight $\omega$) with an elliptic regularization coming from the bending stiffness $b$. It can be considered as a reformulation of the Kirchhoff-Love equations for an elastic rod \cite{landau:b:1970}. For rigorous derivations of such inextensible Kirchhoff beam models from three-dimensional hyper-elasticity see e.g.\ \cite{coleman:p:1993, mora:p:2003}. In non-woven manufacturing the studies of fiber lay-down processes, their longtime behavior and the resulting fabric quality require a fast and accurate numerical treatment, \cite{klar:p:2009, bonilla:p:2007, grothaus:p:2008}. Also nonlinear or even stochastic source terms $\mathbf{f}$ due to aerodynamics might play a role, see \cite{marheineke:p:2011} and Figure~\ref{fig:intro}. So far, the used approaches were mainly addressed to high-speed performance without any theoretical results on convergence or length conservation.

Elastic flows of curves in different model variants were topic of analytical \cite{dziuk:p:2002, oelz:p:2011} and numerical \cite{deckelnick:p:2009, barrett:p:2011, bartels:p:2013} investigations. Considering a global length constraint, an error analysis for a semi-discrete scheme in space was performed in \cite{deckelnick:p:2009}, a fully implicit finite element method with equidistribution properties was explored in \cite{barrett:p:2011}. The work \cite{bartels:p:2013} presented a scheme for an arc-length parameterized curve whose dynamics is caused by bending and friction, neglecting inertia. This is a model system quite similar to \eqref{eq:intro} in the spatial terms, but first order in time and dissipative. The nonlinear point-wise constraint of the local length preservation was handled by a linearization around a previous solution in each time step which led to a sequence of linear saddle-point problems. We adapt this idea to our problem. 

This work aims at the development of a numerical scheme for \eqref{eq:intro} with focus on analytical and computational aspects.
We propose a semi-discretization in time. Following the concept of \cite{juengel:p:2001} and employing a horizontal line method, we replace the transient problem by a sequence of elliptic systems that are handled in their weak formulation in terms of the Lagrange formalism. The algebraic constraint is incorporated in a linearized form in the definition of the optimization domain such that we study the solvability of a constrained minimization problem in a Hilbert space setting \cite{hinze:b:2009, troeltzsch:b:2010}. We prove the existence of the minimizer and of the Lagrange multiplier on each time level. Stability estimates on the discrete solution and the Lagrange multiplier result then in the convergence of the numerical scheme as the time step goes to zero, $\tau\rightarrow 0$ (Theorem~\ref{th:onconv}). In the limit the arc-length constraint is fulfilled. In addition, we derive an explicit error bound of order $\mathcal{O}(\sqrt{\tau})$ on the discrete fiber elongation (Proposition~\ref{estimateconstraint}). Numerically, we solve the optimization problems in finite element spaces. The finite dimensional approximation of the constraint determines the accuracy and efficiency of the scheme. 

The paper is structured as follows. Proceeding from the model system for the inextensible inertial fiber, we present the numerical scheme in Section~\ref{sec:2}. In Section~\ref{sec:3} we deal with its theoretical analysis, regarding existence, stability and convergence. The numerical realization is discussed in Section~\ref{sec:4}. The simulation results illustrate the qualitative behavior of the fiber dynamics and confirm the analytically predicted properties. We conclude with a summary and an outlook.

\begin{figure}[tb]
\vspace*{0.6cm}
\includegraphics[scale=0.65]{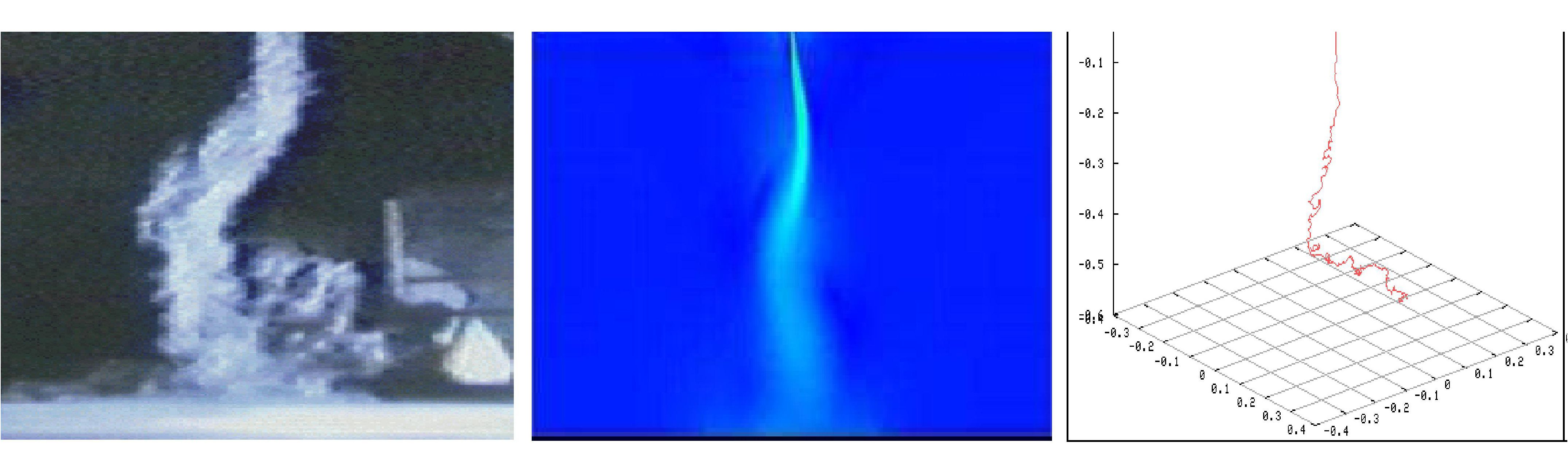}
\caption{\label{fig:intro} {Application: melt-spinning process of non-woven materials. \emph{From left to right:} Turbulent air flow in process (photo by industrial partner), mean velocity flow field, turbulence effects on immersed fiber modeled by stochastic forces (source terms) in  \eqref{eq:intro}, \cite{marheineke:p:2007}.}}
\end{figure}

%%%%%%%%%%%%%%%%%%%%%%%%%%%%%%%%%%%%%%%%%%%%
\section{Numerical scheme for the fiber model}\label{sec:2}
\setcounter{equation}{0}

\subsection{Model}
A fiber is characterized by its long slender geometry. According to the special Cosserat theory \cite{antman:b:2006} it can be asymptotically represented by its arc-length parameterized time-dependent center-line $\mathbf{r}:\Omega=\Omega_L\times \Omega_T \rightarrow \mathbb{R}^3$, where $\Omega_a:=(0,a)$, $a\in(0,\infty)$ with fiber length $L$ and end time $T$. Since extension and shear are here negligibly small in comparison to bending, the dynamics of an homogeneous inertial elastic fiber can be described by a wave-like system of fourth order with constraint 
\begin{align}\label{eq:1}
\omega \,\partial_{tt} \mathbf{r}(s,t) &= \partial_s(\lambda(s,t) \, \partial_s \mathbf{r}(s,t))-b \, \partial_{ssss}\mathbf{r}(s,t) + \mathbf{f}[\mathbf{r},\partial_t\mathbf{r},\partial_s\mathbf{r},s,t], \qquad 
|\partial_s \mathbf{r}(s,t)|^2=1,
\end{align}
where $\omega >0 $ denotes the line weight. The dynamics is caused by the acting inner and outer forces (Newton's law).  The inner force densities stem from bending with bending stiffness $b>0$ as well as from traction $\lambda$. The inner traction $\lambda: \Omega \rightarrow \mathbb{R}$ acts particularly as Lagrange multiplier to the nonlinear point-wise constraint that is expressed in the Euclidian norm $|\cdot|$ and ensures the arc-length parameterization for all times. It enforces the local inextensibilty and hence the global conservation of length.
The neglect of torsion, i.e.\ $\kappa=0$, in the model is justified by respective boundary conditions, such as a free ending or torsion-free clamping. Its inclusion would yield an extra term $\kappa(\partial_s\mathbf{r} \times \partial_{sss}\mathbf{r})$ in the system and the associated equation $\partial_s \kappa=0$, cf.\ \cite{klar:p:2009}. The system~\eqref{eq:1} is a reformulation of the Kirchhoff-Love equations, for details on its derivation we refer to \cite{marheineke:p:2006}. As far as we know there are no existence results for \eqref{eq:1}. The Kirchhoff-Love equations are the limit system of an elastic Euler-Bernoulli rod, as the slenderness parameter (ratio between fiber diameter and length) and the Mach number (ratio between fiber velocity and speed of sound) approach zero \cite{baus:p:2015}. Depending on the application, the outer force densities $\mathbf{f}$ might come for example from gravity, friction or aerodynamics. In case of a linear force in $\partial_t \mathbf{r}$ (e.g.\ friction) and negligible inertia effects, the system~\eqref{eq:1} reduces to an evolution equation (first order in time) with constraint that was subject of research in \cite{oelz:p:2011, bartels:p:2013}. In non-woven manufacturing stochastic effects due to turbulent air flows are important, which implies space-time white noise as driving forces \cite{marheineke:p:2011, marheineke:p:2007} (cf.\ Figure~\ref{fig:intro}). For a study on extensible stochastic beam equations (without constraint) see e.g.\ \cite{brzezniak:p:2005, baur:p:2013}. 

In this work we restrict to sufficiently smooth outer forces that are independent of the fiber curve, like for example gravity. We consider a set-up where a fiber fixed at one ending is freely swinging. Initially, it is assumed to be free of stress and to rest in a straight position. So, the following initial conditions as well as Dirichlet and Neumann boundary conditions for the clamped ($s=L$) and stress-free ($s=0$) fiber ending close \eqref{eq:1} to an initial boundary value problem
\begin{equation}\label{eq:1bd}
\begin{aligned}
\mathbf{r}(s,0)&=(L-s)\mathbf{e_g}, \hspace*{3.2cm} \partial_t \mathbf{r}(s,0)=\mathbf{0}\\
\mathbf{r}(L,t)& = \mathbf{0}, \hspace*{4.4cm} \partial_s \mathbf{r}(L,t)=-\mathbf{e_g}  \\
\partial_{ss}\mathbf{r}(0,t)&= \mathbf{0}, \hspace*{2.5cm}  (b\partial_{sss}\mathbf{r}-\lambda \partial_s\mathbf{r})(0,t)=\mathbf{0}
\end{aligned}
\end{equation}
with the normalized direction vector $\mathbf{e_g}$, $|\mathbf{e_g}|=1$. The natural boundary conditions are equivalent to $\partial_{ss}\mathbf{r}(0,t)=\partial_{sss}\mathbf{r}(0,t)=\mathbf{0}$ and $\lambda(0,t)=0$ under the constraint. Moreover, the inner traction force might consistently satisfy $\lambda(s,0)=0$ for a stress-free initial configuration. This set-up of a cantilever beam reminds on hair modeling in computer graphics. In view of applications in technical textile industry it is a simplification, but it still contains the major mathematical difficulty, i.e.\ the partial differential-algebraic structure of the model equations. 

\begin{assumption}
Let $\mathbf{f}:\overline{\Omega} \to \mathbb{R}^3$ be a continuous function, $0<\omega, b < \infty $ be constants and $\mathbf{e_g}\in\mathbb{R}^3$ be a unit vector for the forthcoming investigations of \eqref{eq:1}-\eqref{eq:1bd}.
\end{assumption}

%%%%%%%%%%%%%%%
\subsection{Semi-discretization} We propose a numerical scheme based on a semi-implicit semi-discreti\-zation. Employing a horizontal line method (Rothe method) in time, we replace the transient problem by a sequence of elliptic systems. The nonlinear arc-length constraint is incorporated in a linearized version.

Let $T\in(0,\infty)$ be given. We divide the time interval $\overline{\Omega}_T=[0,T]$ into $N$ subintervals by introducing the temporal mesh $\{t_k \,|\, k=0,...,N\}$ where $t_k=k\tau$ is prescribed by the time step $\tau=T/N$.  Using an implicit Euler scheme, we discretize the system \eqref{eq:1}-\eqref{eq:1bd} as
\begin{align*}
\frac{\omega}{\tau^2} (\mathbf{r}_{k+1} - 2\mathbf{r}_{k} +\mathbf{r}_{k-1}) 
  & = \partial_s(\lambda_{k+1} \partial_s \mathbf{r}_{k+1})-b\partial_{ssss}\mathbf{r}_{k+1}+\mathbf{f}_{k+1},\qquad \qquad
|\partial_s\mathbf{r}_{k+1}|^2 = 1, 
\end{align*}
with $\mathbf{ f}_{k+1} = \mathbf{ f}(t_{k+1})$, $k=1, \ldots ,N-1$, $\mathbf{ r}_0=(L-s)\mathbf{ e_g}$ and $\mathbf{ r}_1=\mathbf{ r}_0$.
The implicit time discretization requires consequently the recursive solving of nonlinear constrained elliptic systems in one space dimension.
As we will show, it is sufficient to consider the constraint as $\partial_t |\partial_s \mathbf{r}_{k+1}|^2=0$ and express it in terms of the linearization around the solution associated with the previous time step. This yields 
\begin{subequations}\label{eq:2}
\begin{align}\label{eq:ode2}
\frac{\omega}{\tau^2} (\mathbf{ r}_{k+1} - 2\mathbf{ r}_{k} +\mathbf{ r}_{k-1}) 
  & = \partial_s(\lambda_{k+1} \partial_s \mathbf{ r}_{k})-b\partial_{ssss}\mathbf{ r}_{k+1}+\mathbf{ f}_{k+1}\\\label{eq:con}
(\partial_s\mathbf{ r}_{k+1}-\partial_s\mathbf{r}_k)\cdot \partial_s\mathbf{r}_k &= 0. 
\end{align}
\end{subequations}
The approximate solution to \eqref{eq:1} is then given by the linear interpolation $(\mathbf{ r}^{\tau},\lambda^{\tau})$, i.e.,
\begin{align*}
\mathbf{ r}^\tau(s,t)&=\frac{t-t_{k-1}}{\tau} (\mathbf{ r}_k-\mathbf{ r}_{k-1}) +\mathbf{ r}_{k-1},
&& s \in \Omega_L, \quad t\in(t_{k-1},t_k], \quad k=1, \ldots, N, 
\end{align*}
and correspondingly for $\lambda^{\tau}$ which we extend $\lambda_0 = \lambda_1 = 0$ in view of a stress-free initial solution. For functions defined on $[0,T]$, in turn, a subindex $k \in \{0, \ldots, N\}$ corresponds to the value of the function at time $t_k$. The discretized system \eqref{eq:2} can be identified as Euler-Lagrange equations corresponding to an appropriate Lagrange functional, such that we explore its solvability as a variational problem. It will turn out that in the limit $\tau \rightarrow 0$ the system \eqref{eq:1} is fulfilled.

%%%%%%%%%%%%%%%%%%%%%%%%%%%%%%%%%%%%%%%
\section{Theoretical analysis}\label{sec:3}
\setcounter{equation}{0}

In this section we handle the sequence of elliptic systems in their weak formulation in terms of the Lagrange formalism. The constraint is incorporated in the definition of the optimization domain such that we study the solvability of a constrained minimization problem. In particular, we show the existence of the minimizer and of the Lagrange multiplier on each time level. Stability estimates on the discrete solution and the Lagrange multiplier result then in a convergence proof for the numerical scheme.

%%%%%%
\subsection{Solvability of the discretized system}

The norm of a Banach space $\mathcal{B}$ we denote by $\|\cdot\|_\mathcal{B}$ and the dual pairing with its dual space $\mathcal{B}'$ by $_{\mathcal{B}}{\langle\cdot,\cdot\rangle}_{\mathcal{B}'}$. If $\mathcal{B}$ even is a Hilbert space, then its inner product we denote by $(\cdot, \cdot)_\mathcal{B}$. By $\mathcal{L}^2(S; {\mathbb R}^n)$, $S \subset {\mathbb R}^d$ Lebesgue measurable, $n, d \in {\mathbb N}$, we denote the Hilbert space of (equivalence classes of) square integrable functions on $S$ w.r.t.\ the Lebesgue measure taking values in ${\mathbb R}^n$. The space $\mathcal{C}^0(S; \mathcal{B})$ of continuous functions on compact $S$ with values in $\mathcal{B}$ we consider to be equipped with the norm of uniform convergence. We use the notation $\mathcal{W}^{m,p}(U; {\mathbb R}^n)$, $U \subset {\mathbb R}^d$ open, $m\in [0,\infty)$, $p\in [1,\infty]$, for Sobolev spaces as in \cite{adams:b:1990}. In case of $p=2$, the Hilbert space $\mathcal{W}^{m,2}(U; {\mathbb R}^n)$ is abbreviated by $\mathcal{H}^m(U; {\mathbb R}^n)$. In the case $n=1$ we suppress the range of function spaces. 
In particular, we introduce the notation 
\begin{align*}
\mathcal{H}^m_{0,a}(\Omega_a; {\mathbb R}^n):=\{{\mathbf v}\in \mathcal{H}^m(\Omega_a; {\mathbb R}^n) \,|\, \partial_s^\alpha{\mathbf v}(a)={\mathbf 0} \mbox{ for all } \alpha \in \mathbb{N}_0, \alpha + 1/2 < m\}, \quad \Omega_a = (0,a),  
\end{align*}
$a\in (0,\infty)$.
Of course, $\mathcal{H}^m_{0,a}(\Omega_a; {\mathbb R}^n)$ equipped with the norm of $\mathcal{H}^m(\Omega_a; {\mathbb R}^n)$ is a Hilbert space. Its dual space $(\mathcal{H}_{0,a}^m(\Omega_a;{\mathbb R}^n))'$ we denote by $\mathcal{H}^{-m}(\Omega_a;{\mathbb R}^n)$. Recall that $\mathcal{H}^m(\Omega_a; {\mathbb R}^n)$ is embedded continuously and compactly in the H\"{o}lder spaces $\mathcal{C}^{k, \gamma}([0,a];{\mathbb R}^n)$ for $m > 1/2 + k + \gamma$, $k \in {\mathbb N}_0$, $0 \le \gamma \le 1$, see e.g.\ \cite{adams:b:1990}. We always, via the Riesz representation theorem, identify spaces of square integrable functions with their dual space and consider an embedding of Sobolev spaces in the sense of Gelfand triples with the space of square integrable functions as central space.

We define the affine linear fiber space
\begin{align*} 
\mathcal{V}:=\{{\mathbf v}\in {\mathbf v}_D + \mathcal{H}^2_{0,L}(\Omega_L;\mathbb{R}^3) \,|\,\,{\mathbf v}_D \in \mathcal{H}^2(\Omega_L;\mathbb{R}^3),\, {\mathbf v}_D(L)={\mathbf 0}, \, \partial_s{\mathbf v}_D(L)=-\mathbf{e_g}\}
\end{align*}
and introduce the constraint associated functional
\begin{align}\label{def:e}
e_{k+1}: \mathcal{V}\rightarrow \mathcal{H}_{0,L}^1(\Omega_L), \quad e_{k+1}({\mathbf v})=2\partial_s({\mathbf v}-{\mathbf r}_k)\cdot\partial_s{\mathbf r}_k = 0.
\end{align}
Moreover, we deduce the cost functionals $J_{k+1}:\mathcal{V}\rightarrow {\mathbb{R}}$
\begin{align}\label{def:J}
J_{k+1}({\mathbf v}) = {\omega} \left\|\tau {\mathrm D}_{k+1}^2 {\mathbf v}\right\|^2_{\mathcal{L}^2(\Omega_L)} + {b}\, \|\partial_{ss}{\mathbf v}\|^2_{\mathcal{L}^2(\Omega_L)} 
-2 \,({\mathbf f}_{k+1}, {\mathbf v} )_{\mathcal{L}^2(\Omega_L)}
\end{align} 
with the second temporal difference ${\mathrm D}_{k+1}^2{\mathbf v} = ({\mathbf v}-2{\mathbf r}_k+{\mathbf r}_{k-1})/\tau^2$  by applying variational calculus on \eqref{eq:ode2} for $k=1,...,N-1$.

\quad\\
\textbf{Lagrange formalism.} {\it 
For $k=1,...,N-1$, let $J_{k+1}$ be the cost functional of \eqref{def:J} and $e_{k+1}$ the constraint functional of \eqref{def:e}. Define the Lagrange functional
  $L_{k+1}:\mathcal{V}\times\mathcal{H}^{-1}(\Omega_L)\rightarrow \mathbb{R}$ by}
\begin{align*}
 L_{k+1}({\mathbf v},\lambda)= J_{k+1}({\mathbf v}) + {}_{\mathcal{H}_{0,L}^1(\Omega_L)}{\!\langle e_{k+1}({\mathbf v}),\lambda\rangle}_{\mathcal{H}^{-1}(\Omega_L)}.
\end{align*}
{\it Then, a stationary point of the Lagrange functional is a weak solution of the fiber system \eqref{eq:2}.}

\quad\\
A stationary point of the Lagrange functional satisfies the adjoint problem \eqref{eq:ad} for all test functions
$\eta\in\mathcal{H}^{-1}(\Omega_L)$ and $\boldsymbol{\phi}\in \mathcal{H}^2_{0,L}(\Omega_L;\mathbb{R}^3)$, i.e.
\begin{subequations}\label{eq:ad}
\begin{align}\label{eq:adj1}
\partial_\lambda L_{k+1}({\mathbf v},\lambda)[\eta] =0 
&= {}_{\mathcal{H}_{0,L}^1(\Omega_L)}{\!\langle e_{k+1}({\mathbf v}),\eta\rangle}_{\mathcal{H}^{-1}(\Omega_L)}\\\label{eq:adj2}
\nabla_{\mathbf v} L_{k+1}({\mathbf v},\lambda) [\boldsymbol{\phi}] = 0 
&= J_{k+1}'({\mathbf v})[\boldsymbol{ \phi}] + {}_{\mathcal{H}_{0,L}^1(\Omega_L)}{\!\langle e_{k+1}'({\mathbf v})[\boldsymbol{\phi}],\lambda\rangle}_{\mathcal{H}^{-1}(\Omega_L)} \\ \nonumber
&= 2 \, \big( \omega \,( {\mathrm D}_{k+1}^2 {\mathbf v},\boldsymbol{\phi} )_{\mathcal{L}^2(\Omega_L)} 
+b \,( \partial_{ss}{\mathbf v}, \partial_{ss} \boldsymbol{\phi})_{\mathcal{L}^2(\Omega_L)} -({\mathbf f}_{k+1}, \boldsymbol{\phi})_{\mathcal{L}^2(\Omega_L)} \\\nonumber
& \quad \quad 
+  {}_{\mathcal{H}_{0,L}^1(\Omega_L)}{\!\langle \partial_{s}{\mathbf r}_k \cdot\partial_{s}\boldsymbol{\phi},\lambda\rangle}_{\mathcal{H}^{-1}(\Omega_L)}\big).
\end{align}
\end{subequations}
Presupposing sufficient regularity of the Lagrange multiplier $\lambda$, the duality pairing ${}_{\mathcal{H}_{0,L}^1(\Omega_L)}{\!\langle\cdot,\cdot\rangle}_{\mathcal{H}^{-1}(\Omega_L)}$ coincides with $(\cdot,\cdot)_{\mathcal{L}^2(\Omega_L)}$ in the sense of a Gelfand triple. This yields the Euler-Lagrange equations to \eqref{eq:2}. Hence, the weak solvability of the fiber system \eqref{eq:2} can be formulated as

\quad\\
\textbf{Constrained minimization problem}
\begin{align}\label{def:mini}
\textit{Minimize }J_{k+1}\textit{ over the domain } \mathcal{K}_{k+1}:=\{{\bf v} \in \mathcal{V}
\,|\,e_{k+1}({\bf v})=0 \}.
\end{align}

\begin{lemma}[Properties of cost functional]\label{lem:1}
For $k=1,...,N-1$, the cost functional $J_{k+1}:\mathcal{V}\rightarrow{\mathbb R}$ defined in \eqref{def:J} is strictly convex, coercive and weakly lower semi-continuous. The minimization domain $\mathcal{K}_{k+1}$ is closed and convex and, in particular, weakly closed.
\end{lemma}

\begin{proof}
Here and throughout the following proofs where is no danger of confusion, we suppress the indices indicating the time levels for a simpler notation. Let ${\mathbf u}$, ${\mathbf v}\in \mathcal{V}$, ${\mathbf u}\neq {\mathbf v}$, $\mu\in (0,1)$.  Then, the strict convexity of $J$ is concluded from
\begin{align*}
\mu J({\mathbf u})&+(1-\mu)J({\mathbf v}) - J(\mu {\mathbf u} +(1-\mu){\mathbf v}) \\
  & =(\mu-\mu^2) ( \omega \,\|\tau{\mathrm D}^2{\mathbf u}-\tau{\mathrm D}^2{\bf v}\|^2_{\mathcal{L}^2(\Omega_L)} 
  + b\,\|\partial_{ss}{\mathbf u} - \partial_{ss}{\mathbf v}\|^2_{\mathcal{L}^2(\Omega_L)})
  >0
\end{align*}
since $\omega, b>0$.

Due to the assumed boundary conditions a Poincar$\acute{\rm e}$ inequality holds and we obtain
\begin{align*}
\omega \|\tau{\mathrm D}^2{\mathbf v}\|^2_{\mathcal{L}^2(\Omega_L)}+ b\|\partial_{ss}{\mathbf v}\|^2_{\mathcal{L}^2(\Omega_L)} \geq A_1 \|{\mathbf v}\|^2_{\mathcal{H}^2(\Omega_L)} - A_2
\end{align*}
for some $0 < A_1, A_2 < \infty$.
Hence
\begin{align*}
J({\mathbf v}) &= \omega \|\tau{\mathrm D}^2{\mathbf v}\|^2_{\mathcal{L}^2(\Omega_L)}+ b\|\partial_{ss}{\mathbf v}\|^2_{\mathcal{L}^2(\Omega_L)} 
  - 2 \langle { {\mathbf f}},{\mathbf v}\rangle_{\mathcal{L}^2(\Omega_L)}
  \geq A_1 \|{\mathbf v}\|^2_{\mathcal{H}^2(\Omega_L)} - 2|\langle {\mathbf f}, {\mathbf v}\rangle_{\mathcal{L}^2(\Omega_L)}| - A_2\\
& \geq \|{\mathbf v}\|_{\mathcal{H}^2(\Omega_L)}\, (A_1\|{\mathbf v}\|_{\mathcal{H}^2(\Omega_L)}-2\|{\mathbf f}\|_{\mathcal{L}^2(\Omega_L)}) - A_2.
\end{align*}
Thus, $J({\mathbf v})\rightarrow \infty$, if $\|{\mathbf v}\|_{\mathcal{H}^2(\Omega_L)} \rightarrow \infty$ for fixed ${\mathbf f}\in \mathcal{L}^2(\Omega_L)$, i.e., $J$ is coercive.
  
Let $({\mathbf v}_n)_{n\in \mathbb{N}}$ be a sequence in $\mathcal{V}$ that converges weakly to ${\mathbf v}\in \mathcal{V}$ in ${\mathcal{H}}^2$, 
i.e., ${\mathbf v}_n \stackrel{\mathcal{H}^2}{\rightharpoonup} {\mathbf v}$ for $n\rightarrow \infty$. 
Then, in particular, ${\mathbf v}_n \stackrel{\mathcal{L}^2}{\rightharpoonup} {\mathbf v}$ and 
$\partial_{ss}{\mathbf v}_n \stackrel{\mathcal{L}^2}{\rightharpoonup} \partial_{ss}{\mathbf v}$ for $n\rightarrow \infty$. 
Since the norm is lower semi-continuous w.r.t.\ weak convergence and the inner product with ${\mathbf f}\in \mathcal{L}^2(\Omega_L)$ is continuous w.r.t.\ weak convergence, we obtain $J({\mathbf v})\leq \lim_{n\rightarrow \infty} \inf J({\mathbf v}_n)$, i.e., $J$ is weakly lower semi-continuous.  

The convexity of $\mathcal{K}$ results from the affine linearity of $e$.
Let ${\mathbf u}$, ${\mathbf v}\in \mathcal{K}$, $\mu\in [0,1]$. Then, it holds $\mu{\mathbf u}+(1-\mu){\mathbf v} \in \mathcal{K}$ because of
$e(\mu {\mathbf u}+(1-\mu){\mathbf v}) = \mu e({\mathbf u})+(1-\mu)e({\mathbf v})=0$.

Since $\mathcal{V}$ is closed and $e$ is continuous, also $\mathcal{K}$ is closed. This,
together with convexity, implies that $\mathcal{K}$ is also weakly closed.
\end{proof}

\begin{theorem}[Existence and uniqueness of minimizer]\label{theo:mini}
The constrained minimization problem \eqref{def:mini} has a unique solution ${\mathbf r}_{k+1}\in \mathcal{K}_{k+1}$ on every time level $k = 1, \ldots, N-1$.
\end{theorem}

\begin{proof}
Lemma~\ref{lem:1} provides the necessary conditions for a general existence and uniqueness result for constrained minimization problems, see e.g.~\cite{troeltzsch:b:2010}. We state the proof here for completeness.

Choose a minimizing sequence $({\mathbf v}_n)_{n\in \mathbb{N}}$, ${\mathbf v}_n\in \mathcal{K}$, with $J({\mathbf v}_n) \rightarrow \inf_{{\mathbf v}\in \mathcal{K}} J({\mathbf v})$ for $n\rightarrow \infty$. Then $-\infty < \inf_{{\mathbf v}\in \mathcal{K}} J({\mathbf v})< \infty$ and $({\mathbf v}_n)_{n\in \mathbb{N}}$  is bounded in view of the coercivity of $J$. Hence, there exists a subset $D\subset \mathbb{N}$ and ${\mathbf r}\in \mathcal{V}$ such that ${\mathbf v}_n\stackrel{\mathcal{H}^2}{\rightharpoonup}{\mathbf r}$ for $D \ni n \to \infty$.  Since $\mathcal{K}$ is weakly closed, ${\mathbf r}\in \mathcal{K}$. The weak lower semi-continuity of $J$ implies 
$J({\mathbf r}) \leq \inf_{n\in D} J({\mathbf v}_n)$, whence ${\mathbf r}$ is a minimizer.

Since $\mathcal{K}$ is convex, the strict convexity of $J$ on $\mathcal{V}$ implies the uniqueness of the minimizer. Assume ${\mathbf u},{\mathbf v}\in \mathcal{K}$ 
to be two minimizers that satisfy ${\mathbf u}\neq {\mathbf v}$ with $J({\mathbf u})=J({\mathbf v})$.
Then $J(\mu {\mathbf u}+(1+\mu){\mathbf v})< \mu J({\mathbf u})+(1-\mu)J({\mathbf v})=J({\mathbf u})$ for $\mu\in (0,1)$.
Since $\mu {\mathbf u}+(1+\mu){\mathbf v}\in \mathcal{K}$ for $\mu\in (0,1)$, this contradicts the assumption.
\end{proof}

Note that the uniqueness of the minimizer is meaningless for the solvability statement of the fiber system, since the unique minimizer need not necessarily be the only solution in view of possibly existing saddle points.

The fact $e_{k+1}({\mathbf r}_{k+1}) = 0$ implies
\begin{align*}
0 \leq |\partial_s{\mathbf r}_{k+1} - \partial_s{\mathbf r}_{k}|^2 = |\partial_s{\mathbf r}_{k+1}|^2 + |\partial_s{\mathbf r}_{k}|^2
- 2\partial_s{\mathbf r}_{k+1} \cdot \partial_s{\mathbf r}_{k} = |\partial_s{\mathbf r}_{k+1}|^2 - |\partial_s{\mathbf r}_{k}|^2
\end{align*}
for all $k=0,...,N-1$. Hence, together with $|\partial_s{\bf r}_{0}| = 1$, the following lemma follows immediately.

\begin{lemma}\label{le:increase}
The relation $1\le |\partial_s{\bf r}_{k}(s)| \le |\partial_s{\bf r}_{k+1}(s)|$ holds for all $s \in \Omega_L$, $k=0,...,N-1$.
\end{lemma}

\begin{proposition}[Surjectivity of linearized constraint functional] \label{th:exLM}
For $k=1,...,N-1$ the linearized constraint functional $e_{k+1}'\in \mathcal{L}(\mathcal{H}^2_{0,L}(\Omega_L; {\mathbb R}^3);\mathcal{H}^1_{0,L}(\Omega_L))$ is surjective.
\end{proposition}

\begin{proof} 
We have $e'_{k+1}[\boldsymbol{\phi}] = 2  \partial_s{\mathbf r}_k \cdot \partial_s \boldsymbol{\phi}$, $\boldsymbol{\phi}\in \mathcal{H}^2_{0,L}(\Omega_L)$.
Let $\psi\in \mathcal{H}^1_{0,L}(\Omega_L)$ be arbitrary. Set
\begin{align}\label{eq:inverse}
\boldsymbol{\phi}(s) = -\int_s^L \frac{\psi \,\partial_u {\mathbf r}_k}{2|\partial_u{\mathbf r}_k|^2} \,du, \quad s \in \Omega_L.
\end{align}
Note that $\partial_{s} {\mathbf r}_k, \psi\in \mathcal{C}^0([0,L])$. Hence, together with Lemma \ref{le:increase}, 
we obtain $\boldsymbol{\phi} \in \mathcal{L}^2(\Omega_L; {\mathbb R}^3)$ with $\boldsymbol{\phi}(L) = 0$. 
We find
\begin{align*}
\psi= 2 \partial_s{\mathbf r}_k \cdot \partial_s\boldsymbol{\phi} = e'_{k+1}[\boldsymbol{\phi}], \qquad \qquad \text{as }\partial_s \boldsymbol{\phi} = \frac{\psi \,\partial_s {\mathbf r}_k}{2|\partial_s{\mathbf r}_k|^2}
\end{align*}
holds. Moreover $\partial_s \boldsymbol{\phi} \in \mathcal{L}^2(\Omega_L; {\mathbb R}^3)$ with $\partial_s\boldsymbol{\phi}(L) = 0$. 
Finally, $\partial_{ss}\boldsymbol{\phi}\in \mathcal{L}^2(\Omega_L; {\mathbb R}^3)$ can be concluded from
$\partial_{ss} {\mathbf r}_k \in \mathcal{L}^2(\Omega_L; {\mathbb R}^3)$,
$\partial_s \psi, \psi\in \mathcal{L}^2(\Omega_L)$, $\partial_{s}{\mathbf r}_k \in \mathcal{L}^\infty(\Omega_L; {\mathbb R}^3)$ and
$\psi\in \mathcal{L}^\infty(\Omega_L)$ using the chain rule. This shows $\boldsymbol{\phi}\in \mathcal{H}^2_{0,L}(\Omega_L)$.
\end{proof}

\begin{remark}\label{rm:inverse}
Note that $e_{k+1}'$ is not injective on $\mathcal{H}^2_{0,L}(\Omega_L; {\mathbb R}^3)$. 
Nevertheless we denote the mapping $\psi \mapsto \boldsymbol{\phi}$ in \eqref{eq:inverse} by $(e'_{k+1})^{-1}$, because $e'_{k+1}(e'_{k+1})^{-1}$ is the identity on $\mathcal{H}^1_{0,L}(\Omega_L)$. Of course, $(e'_{k+1})^{-1} \in \mathcal{L}(\mathcal{H}^1_{0,L}(\Omega_L); \mathcal{H}^2_{0,L}(\Omega_L; {\mathbb R}^3))$ by the inverse mapping theorem (applied in the proper quotient space setting).
\end{remark}

\begin{theorem}[Existence of discrete solution]\label{th:exist}
For $k=1,...,N-1$ let ${\mathbf r}_{k+1}$ be the minimizer of $J_{k+1}$ on $\mathcal{K}_{k+1}$, provided in Theorem \ref{theo:mini}, and
\begin{align}\label{eq:lambda}
\lambda_{k+1} := -J_{k+1}'({\mathbf r}_{k+1})(e'_{k+1})^{-1}.
\end{align}
Then $\lambda_{k+1} \in \mathcal{H}^{-1}(\Omega_L)$ (see Remark \ref{rm:inverse}), and $({\mathbf r}_{k+1}, \lambda_{k+1})$ are solving weakly the discrete fiber system \eqref{eq:2}, i.e., 
\begin{subequations}\label{eq:exist}
\begin{align}\label{eq:exist1}
\omega \,({\mathrm D}_{k+1}^2 {\mathbf r}_{k+1},\boldsymbol{\phi} )_{\mathcal{L}^2(\Omega_L)} 
&= - {}_{\mathcal{H}_{0,L}^1(\Omega_L)}{\!\langle \partial_{s}{\mathbf r}_k \cdot\partial_{s}\boldsymbol{\phi},\lambda_{k+1}\rangle}_{\mathcal{H}^{-1}(\Omega_L)} \nonumber\\
& \quad - b \,( \partial_{ss}{\bf r}_{k+1}, \partial_{ss} \boldsymbol{\phi})_{\mathcal{L}^2(\Omega_L)} +({\mathbf f}_{k+1}, \boldsymbol{\phi})_{\mathcal{L}^2(\Omega_L)} \\\label{eq:exist2}
\partial_s ({\mathbf r}_{k+1}- {\mathbf r}_{k}) \cdot \partial_s {\mathbf r}_{k} &= 0. 
\end{align}
\end{subequations}
for all test functions $\boldsymbol{\phi}\in \mathcal{H}^2_{0,L}(\Omega_L;\mathbb{R}^3)$.
\end{theorem}

\begin{proof}
By definition, all elements from $\mathcal{K}_{k+1}$ fulfill \eqref{eq:exist2}. Furthermore, since ${\bf r}_{k+1}$ minimizes $J_{k+1}$ on $\mathcal{K}_{k+1}$ and
\begin{align*}
e_{k+1}[\boldsymbol{\phi} - (e'_{k+1})^{-1}[2  \partial_s{\bf r}_k \cdot \partial_s \boldsymbol{\phi}]] = 0 \quad \mbox{ for all } \boldsymbol{\phi}\in \mathcal{H}^2_{0,L}(\Omega_L;\mathbb{R}^3),
\end{align*}
we have
\begin{align*}
0 &= J_{k+1}'({\mathbf r}_{k+1}) [\boldsymbol{ \phi} - (e'_{k+1})^{-1}[2  \partial_s{\mathbf r}_k \cdot \partial_s \boldsymbol{\phi}]]\\
&= 2 \, \big( \omega \,( {\mathrm D}_{k+1}^2 {\mathbf r}_{k+1}, \boldsymbol{\phi} )_{\mathcal{L}^2(\Omega_L)}
   + b \,( \partial_{ss}{\mathbf r}_{k+1}, \partial_{ss} \boldsymbol{\phi} )_{\mathcal{L}^2(\Omega_L)} 
   -( {\mathbf  f}_{k+1}, \boldsymbol{\phi} )_{\mathcal{L}^2(\Omega_L)} \\
& \quad   \quad +  {}_{\mathcal{H}_{0,L}^1(\Omega_L)}{\!\langle \partial_{s}{\bf r}_k \cdot\partial_{s}\boldsymbol{\phi},\lambda_{k+1}\rangle}_{\mathcal{H}^{-1}(\Omega_L)}\big) 
&& \mbox{ for all }\boldsymbol{\phi}\in \mathcal{H}^2_{0,L}(\Omega_L;\mathbb{R}^3).
\end{align*}
\end{proof}

%%%%%%%
\subsection{Stability estimates}

In the following we provide stability estimates for the discrete solution. For function spaces $\mathcal{B}_1(S_1)$ and $\mathcal{B}_2(S_2)$ on sets $S_1$ and $S_1$, respectively, we define as usual $\mathcal{B}_1(S_1) \otimes \mathcal{B}_2(S_2):= \mathrm{span}\{ f_1 f_2 \,| \, f_1 \in \mathcal{B}_1(S_1), f_2 \in \mathcal{B}_2(S_2)\}$, where
$(f_1 f_2)(s_1,s_2) := f_1(s_1)f_2(s_2)$, $s_1 \in S_1$, $s_2 \in S_2$ (algebraic tensor product).
The Sobolev space $\mathcal{H}^{2,1}(\Omega;\mathbb{R}^3)$ on $\Omega := \Omega_L \times \Omega_T$ 
then is defined as the completion of $\mathcal{H}^{2}(\Omega_L;\mathbb{R}^3) \otimes \mathcal{H}^{1}(\Omega_T)$ w.r.t.\ the metric associated to its inner product, see e.g.~\cite[Chap.~II.4]{reed:b:1980} (i.e., it is the Sobolev space of functions on $\Omega$ which are twice weakly differentiable in the first variable and once weakly differentiable in the second variable and square integrable on $\Omega$ together with their derivatives). Correspondingly, we set
$\mathcal{H}_{0,L,T}^{m_L, m_T}(\Omega)$, $m_L, m_T \in [0, \infty)$, to be the completion of $\mathcal{H}_{0,L}^{m_L}(\Omega_L) \otimes \mathcal{H}_{0,T}^{m_T}(\Omega_T)$ and use the notation $\mathcal{H}^{-m_L,-m_T}(\Omega) := (\mathcal{H}_{0,L,T}^{m_L, m_T}(\Omega))'$. In the case $m_L=m_T$ we suppress the index $m_T$.
For functions $h$ defined on $\Omega_T$ we use the following notation for discrete derivatives:
\begin{align*}
({\mathrm D}h)_{k} = ({\mathrm D}^1h)_{k} := \frac{h_{k} - h_{k-1}}{\tau}, \quad ({\mathrm D}^{n}h)_{k} := ({\mathrm D}({\mathrm D}^{n-1}h)^\tau)_k,
\quad k = n, \ldots, N. 
\end{align*}

\begin{proposition}[Stability estimates for ${\mathbf r}^\tau$]\label{stabestir}
Let ${\mathbf r}_{k+1}\in \mathcal{V}$ be as in Theorem \ref{th:exist}, $k=1,...,N-1$, and let ${\mathbf r}^{\tau} \in \mathcal{H}^{2,1}(\Omega;\mathbb{R}^3)$ be the corresponding linear interpolation. Then there exists $0 < K < \infty$, independent of $N \in {\mathbb N}$ (or, equivalently, the time discretization $\tau > 0$), such that
\begin{align}\label{eq:boundonr}
\max_{1 \le k \le N}\|({\mathrm D}{\mathbf r}^\tau)_k\|_{\mathcal{L}^{2}(\Omega_L)}\le K, \qquad 
\max_{0 \le k \le N}\|\partial_{ss}{\mathbf r}_k\|_{\mathcal{L}^{2}(\Omega_L)}\le K, \\ \label{eq:boundonfullr}
\quad \|{\mathbf r}^\tau\|_{\mathcal{H}^{2,1}(\Omega)}\le K,
\qquad \mbox{and  } \,\tau \|\partial_{ss}\partial_t {\mathbf r}^\tau\|^2_{\mathcal{L}^2(\Omega)} \le K.
\end{align}
\end{proposition}

\begin{proof} 
Since ${\mathbf r}_{k+1}\in \mathcal{V} \subset \mathcal{H}^{2}(\Omega_L;\mathbb{R}^3)$ for all $k=1,...,N-1$ and ${\mathbf r}^{\tau}(s)$ is piecewise linear for all $s \in \Omega_L$, we have ${\mathbf r}^{\tau} \in \mathcal{H}^{2,1}(\Omega;\mathbb{R}^3)$. We know that
\begin{align*}
\omega \,( ({\mathrm D}^2{\mathbf r}^\tau)_{k+1},\boldsymbol{\phi} )_{\mathcal{L}^2(\Omega_L)} 
=&- {}_{\mathcal{H}^1(\Omega_L)}{\!\langle \partial_{s}{\mathbf r}_k \cdot\partial_{s}\boldsymbol{\phi}, \lambda_{k+1}\rangle}_{\mathcal{H}^{-1}(\Omega_L)} \\
&- b \,( \partial_{ss}{\mathbf r}_{k+1}, \partial_{ss} \boldsymbol{\phi})_{\mathcal{L}^2(\Omega_L)} +({\mathbf  f}_{k+1}, \boldsymbol{\phi})_{\mathcal{L}^2(\Omega_L)},
\end{align*}
for all $\boldsymbol{\phi}\in \mathcal{H}^2_{0,L}(\Omega_L;\mathbb{R}^3)$.
Note that the first summand on the right-hand side is discretized in an explicit way.
Since $({\mathrm D}^2{\mathbf r}^\tau)_{k+1}= (({\mathrm D}{\mathbf r}^\tau)_{k+1}- ({\mathrm D}{\mathbf r}^\tau)_{k}) /\tau$,
the special choice $\boldsymbol{\phi}={\mathbf r}_{k+1}-{\mathbf r}_k\in \mathcal{H}^2_{0,L}(\Omega_L;\mathbb{R}^3)$ results in
\begin{align}\label{eq:verycrucial}
&\omega \left(\left \|({\mathrm D}{\mathbf r}^\tau)_{k+1} \right \|^2_{\mathcal{L}^2(\Omega_L)}
- \left \|({\mathrm D}{\mathbf r}^\tau)_{k} \right\|^2_{\mathcal{L}^2(\Omega_L)}
+ \|\tau ({\mathrm D}^2{\mathbf r}^\tau)_{k+1}\|^2_{\mathcal{L}^2(\Omega_L)}\right)\\\nonumber
&= -b \left(\|\partial_{ss}{\mathbf r}_{k+1}\|^2_{\mathcal{L}^2(\Omega_L)}- \|\partial_{ss}{\mathbf r}_{k}\|^2_{\mathcal{L}^2(\Omega_L)}
+ \|\partial_{ss}({\mathbf r}_{k+1}-{\mathbf r}_k)\|^2_{\mathcal{L}^2(\Omega_L)}\right)
+2({\mathbf f}_{k+1}, {\mathbf r}_{k+1}-{\mathbf r}_k)_{\mathcal{L}^2(\Omega_L)} 
\end{align}
by applying the identity $2(a-b)a=a^2 - b^2 +(a-b)^2$ and the functional constraint $e_{k+1}({\mathbf r}_{k+1})=0$.
Hence, we obtain
\begin{multline*}
\omega \|({\mathrm D}{\mathbf r}^{\tau})_{k+1}\|^2_{\mathcal{L}^2(\Omega_L)}+
b  \|\partial_{ss}{\mathbf r}_{k+1}\|^2_{\mathcal{L}^2(\Omega_L)} \\ 
\le \omega  \|({\mathrm D}{\mathbf r}^{\tau})_{k}\|^2_{\mathcal{L}^2(\Omega_L)}
+ b \|\partial_{ss}{\mathbf  r}_{k}\|^2_{\mathcal{L}^2(\Omega_L)}
+ 2\tau ({\mathbf f}_{k+1}, ({\mathrm D}{\mathbf r}^\tau)_{k+1})_{\mathcal{L}^2(\Omega_L)}.
\end{multline*}
Summing up $k = 1, \ldots, M-1 \le N-1$ gives the following crucial relation
\begin{align}\label{eq:crucial}
\omega \|({\mathrm D}{\mathbf r}^\tau)_{M}\|^2_{\mathcal{L}^2(\Omega_L)}+ b  \|\partial_{ss}{\mathbf r}_{M}\|^2_{\mathcal{L}^2(\Omega_L)} 
\leq 2\tau \sum_{k=1}^{M-1}({\mathbf f}_{k+1}, ({\mathrm D}{\mathbf r}^\tau)_{k+1})_{\mathcal{L}^2(\Omega_L)},
\end{align}
(note that $({\mathrm D}{\mathbf r}^{\tau})_{1} = {\mathbf 0}$, as well as, $\partial_{ss}{\mathbf r}_{1} = {\mathbf 0}$).
We estimate the scalar product on the right-hand side by Cauchy--Schwarz and Young's inequality, i.e.\ $2ab\leq a^2+b^2$, and find
\begin{align*}
\| ({\mathrm D}{\mathbf r}^\tau)_{M}\|^2_{\mathcal{L}^2(\Omega_L)}
\leq \frac{\tau}{\omega} \left( \sum_{k=1}^{N-1}\|{\mathbf  f}_{k+1}\|^2_{\mathcal{L}^2(\Omega_L)} 
+ \sum_{k=1}^{M-1} \| ({\mathrm D}{\mathbf r}^\tau)_{k+1}\|^2_{\mathcal{L}^2(\Omega_L)} \right).
\end{align*}
The discrete Gronwall Lemma implies 
\begin{align}\label{eq:gron}
\| ({\mathrm D}{\mathbf r}^\tau)_{M}\|^2_{\mathcal{L}^2(\Omega_L)}
\leq \frac{\tau}{\omega} \sum_{k=1}^{N-1} \|{\mathbf f}_{k+1}\|^2_{\mathcal{L}^2(\Omega_L)} \exp\left(\frac{T}{\omega}\right).
\end{align}
Together with 
\begin{align*}
\lim_{N \to \infty} \tau \sum_{k=1}^{N-1}\|{\mathbf f}_{k+1}\|^2_{\mathcal{L}^2(\Omega_L)} =  \|{\mathbf f}\|^2_{\mathcal{L}^2(\Omega)},
\end{align*}
\eqref{eq:gron} yields the existence of $0 < K_1 < \infty$, independent of $N \in {\mathbb N}$, such that
\begin{align}\label{eq:partialt}
\|({\mathrm D}{\mathbf r}^\tau)_M\|^2_{\mathcal{L}^{2}(\Omega_L)}\le K_1.
\end{align}
Combining \eqref{eq:crucial} and \eqref{eq:partialt} gives finally the existence of 
$0 < K_2 < \infty$, independent of $N \in {\mathbb N}$, such that
\begin{align}\label{eq:partialss}
\|(\partial_{ss}{\mathbf r}^\tau)_M\|^2_{\mathcal{L}^{2}(\Omega_L)}\le K_2.
\end{align}
The inequalities \eqref{eq:partialt}, \eqref{eq:partialss} together with the
Poincar$\acute{\rm e}$ inequality guarantee the existence of the desired $0 < K < \infty$, independent of $N \in {\mathbb N}$,
in \eqref{eq:boundonr} and also in the first case of \eqref{eq:boundonfullr}. Finally, summing up \eqref{eq:verycrucial}
in $k=1, \ldots, N-1$ yields:
\begin{multline*}
\omega \tau^2 \sum_{k=1}^{N-1} \|({\mathrm D}^2{\mathbf r}^\tau)_{k+1}\|^2_{\mathcal{L}^2(\Omega_L)}
+ b \tau \|\partial_{ss}\partial_t {\mathbf r}^\tau\|^2_{\mathcal{L}^2(\Omega)} \\
\le 2 \tau \sum_{k=1}^{N-1} ({\mathbf f}_{k+1}, ({\mathrm D}{\mathbf r}^{\tau})_{k+1})_{\mathcal{L}^2(\Omega_L)}
\le 2 \sqrt{K_1} \sqrt{\tau \sum_{k=1}^{N-1} \|{\mathbf f}_{k+1}\|^2_{\mathcal{L}^2(\Omega_L)}}.
\end{multline*}
Thus, also the second estimate in \eqref{eq:boundonfullr} holds true.
\end{proof}

Ideas for proving the next proposition we got from \cite{bartels:p:2013}, where the elastic non-inertial flow (first order in time) of inextensible curves was considered.

\begin{proposition}[Estimates for the algebraic constraint]\label{estimateconstraint}
Let ${\mathbf r}_{k+1}\in \mathcal{V}$ be as in Theorem \ref{th:exist}, $k=1,...,N-1$, and let ${\mathbf r}^{\tau} \in \mathcal{H}^{2,1}(\Omega;\mathbb{R}^3)$ be the corresponding linear interpolation. Then there exists $0 < R < \infty$, independent of $N \in {\mathbb N}$ (or, equivalently, the time discretization $\tau > 0$), such that
\begin{eqnarray}\label{esticonstraint}
\int_0^L |(\partial_{s}{\mathbf r}^{\tau})_{N}|^2 - 1 \,ds \le R \sqrt{\tau}.
\end{eqnarray}
\end{proposition}

\begin{proof} 
For $k=0, \ldots,N-1$ the functional constraint $e_{k+1}({\mathbf r}_{k+1})=0$ implies
\begin{eqnarray}\label{eq:binomi}
|(\partial_{s}{\mathbf r}^{\tau})_{k+1}|^2 = |(\partial_{s}{\mathbf r}^{\tau})_{k} + \tau({\mathrm D}\partial_{s}{\mathbf r}^{\tau})_{k+1}|^2
= |(\partial_{s}{\mathbf r}^{\tau})_{k}|^2 + \tau^2|(\partial_{s}\partial_{t}{\mathbf r}^{\tau})_{k+1}|^2.
\end{eqnarray}
Summation over $k = 0, \dots, N-1$ in \eqref{eq:binomi} yields
\begin{eqnarray*}
|(\partial_{s}{\mathbf r}^{\tau})_{N}|^2 - 1 = \tau^2\sum_{k=0}^{N-1}|(\partial_{s}\partial_{t}{\mathbf r}^{\tau})_{k+1}|^2,
\end{eqnarray*}
because $|(\partial_{s}{\mathbf r}^{\tau})_{0}| = 1$. Therefore,
\begin{multline*}
\int_0^L |(\partial_{s}{\mathbf r}^{\tau})_{N}|^2 - 1 \,ds = \tau^2\sum_{k=0}^{N-1}
\|(\partial_{s}\partial_{t}{\mathbf r}^{\tau})_{k+1}\|^2_{\mathcal{L}^2(\Omega_L)}
= \tau \int_0^T\|(\partial_{s}\partial_{t}{\mathbf r}^{\tau})\|^2_{\mathcal{L}^2(\Omega_L)} \,dt \\
\le R_1 \sqrt{\tau} \left( \sqrt{\int_0^T\|(\partial_{t}{\mathbf r}^{\tau})\|^2_{\mathcal{L}^2(\Omega_L)} \,dt}
\sqrt{\tau \int_0^T\|(\partial_{ss}\partial_{t}{\mathbf r}^{\tau})\|^2_{\mathcal{L}^2(\Omega_L)} \,dt}
+ \sqrt{\tau} \int_0^T\|(\partial_{t}{\mathbf r}^{\tau})\|^2_{\mathcal{L}^2(\Omega_L)} \,dt \right),
\end{multline*}
where we used the Gagliardo--Nirenberg interpolation inequality $\|\partial_{s} v \|^2_{\mathcal{L}^2(\Omega_L)}
\le R_1 \big{(}\|\partial_{ss} v \|_{\mathcal{L}^2(\Omega_L)} + \|v \|_{\mathcal{L}^2(\Omega_L)}\big{)}\| v \|_{\mathcal{L}^2(\Omega_L)}$ holding for some
$0<R_1<\infty$, independent of $v \in \mathcal{H}^2(\Omega_L; {\mathbb R}^3)$. Now the existence of the $0 < R < \infty$ as
in \eqref{esticonstraint}, independent of $\tau > 0$, follows by the second estimate in \eqref{eq:boundonfullr}.
\end{proof}

Before providing the stability bound for the Lagrange multiplier, we need to establish some bounds on the
inverse of the linearized constraint functional.

\begin{lemma}
Let $(e'_{k+1})^{-1} \in \mathcal{L}(\mathcal{H}^1_{0,L}(\Omega_L); \mathcal{H}^2_{0,L}(\Omega_L; {\mathbb R}^3))$ be as in Remark \ref{rm:inverse}, $k=1, \ldots,N-1$. 
Then there exists $0 < M < \infty$, independent of $N \in {\mathbb N}$, such that
\begin{align}\label{estinverse} \nonumber
\|(e'_{k+1})^{-1}[g]\|_{\mathcal{H}^2(\Omega_L)} \le M \, \|g\|_{\mathcal{H}^1(\Omega_L)}, \quad \quad \quad \quad 
\|(e'_{k+1})^{-1}[g]\|_{\mathcal{L}^2(\Omega_L)} \le M \, \|g\|_{\mathcal{L}^2(\Omega_L)}, \\ 
\left\|\big({\mathrm D}\big((e')^{-1}\big)^\tau\big)_{k+1}[g] \right\|_{\mathcal{L}^2(\Omega_L)}
\le M \, \|g\|_{\mathcal{H}^1(\Omega_L)},
\end{align}
for all $g \in \mathcal{H}^{1}_{0,L}(\Omega_L)$.
\end{lemma}

\begin{proof} 
Let $g \in \mathcal{H}^1_{0,L}(\Omega_L)$. As consequence of the Poincar$\acute{\rm e}$ inequality,
Lemma \ref{le:increase} and \eqref{eq:boundonr}, we find constants $0 < M_1, M_2, M_3 < \infty$
(independent of $N \in {\mathbb N}$) such that
\begin{align*}
\|(e'_{k+1})^{-1}[g]\|_{\mathcal{H}^2(\Omega_L)} 
&\le M_1 \left\| \partial_s \left(\frac{g \,\partial_s {\mathbf r}_k}{|\partial_s{\mathbf r}_k|^2}\right) \right\|_{\mathcal{L}^2(\Omega_L)}
\le M_1 \left(\|\partial_s g \,\partial_s {\mathbf r}_k \|_{\mathcal{L}^2(\Omega_L)}
+ 3\| g \,\partial_{ss} {\mathbf r}_k \|_{\mathcal{L}^2(\Omega_L)}\right) \\
&\le M_2 \left(\|\partial_s {\mathbf r}_k \|_{\mathcal{C}^0(\overline{\Omega}_L)}
+ 3\|\partial_{ss} {\mathbf r}_k \|_{\mathcal{L}^2(\Omega_L)}\right) \|g\|_{\mathcal{H}^1(\Omega_L)}
\le M_3 \|g\|_{\mathcal{H}^1(\Omega_L)}.
\end{align*}
Using Cauchy--Schwarz inequality and Lemma \ref{le:increase}, we obtain  
\begin{equation*}
\|(e'_{k+1})^{-1}[g]\|^2_{\mathcal{L}^2(\Omega_L)} 
= \int_0^L \left|\int_u^L \frac{g \,\partial_s {\mathbf r}_k}{2|\partial_s{\mathbf r}_k|^2} \,ds \right|^2 du
\le \frac{3L^2}{4} \int_0^L \frac{g^2 |\partial_s {\mathbf r}_k|^2}{|\partial_s{\mathbf r}_k|^4} \,ds
\le \frac{3L^2}{4} \|g\|^2_{\mathcal{L}^2(\Omega_L)}.
\end{equation*}
The estimation of the discrete derivative is a little more lengthly. Integration by parts yields
\begin{align*}
\Big\| \big({\mathrm D}\big((e')^{-1}&\big)^\tau\big)_{k+1}[g] \Big\|^2_{\mathcal{L}^2(\Omega_L)}
= \int_0^L \left|\int_u^L \frac{g \,\partial_s {\mathbf r}_{k}}{2\tau|\partial_s{\mathbf r}_{k}|^2} 
  - \frac{g \,\partial_s {\mathbf r}_{k-1}}{2\tau|\partial_s{\mathbf r}_{k-1}|^2} \,ds \right|^2 du \\
& \le \int_0^L \left|\int_u^L \frac{g \,\partial_s({\mathbf r}_{k}-{\mathbf r}_{k-1})} {\tau|\partial_s{\mathbf r}_{k}|^2} \,ds \right|^2 du
  + \int_0^L \left|\int_u^L g \,\partial_s{\mathbf r}_{k-1} \left(\frac{1}{\tau|\partial_s{\mathbf r}_{k}|^2} 
  - \frac{1}{\tau|\partial_s{\mathbf r}_{k-1}|^2}\right) \,ds \right|^2 du \\
& \le 2\int_0^L \left|\int_u^L \frac{{\mathbf r}_{k}-{\bf r}_{k-1}}{\tau}\,\,\partial_s\!\left(\frac{g}{|\partial_s{\mathbf r}_{k}|^2}\right)\,ds \right|^2 du
  + 2\int_0^L \left| \frac{{\mathbf r}_{k}-{\mathbf r}_{k-1}} {\tau} \frac{g}{|\partial_s{\mathbf r}_{k}|^2} \right|^2 ds \\
& \quad +\int_0^L \left|\int_u^L \frac{g \,\partial_s{\mathbf r}_{k-1}} {\tau|\partial_s{\mathbf r}_{k}|^2|\partial_s{\mathbf r}_{k-1}|^2}
  \,\,\partial_s({\mathbf r}_{k}-{\mathbf r}_{k-1})\cdot\partial_s({\mathbf r}_{k}+{\mathbf r}_{k-1}) \,ds \right|^2 du \\
& \le M_4 \,\|g\|^2_{\mathcal{H}^1(\Omega_L)} +2\int_0^L \left|\int_u^L \frac{{\mathbf r}_{k}
  -{\mathbf r}_{k-1}}{\tau} \cdot \partial_s\!\left(\partial_s({\mathbf r}_{k}+{\mathbf r}_{k-1})\frac{g \,\partial_s{\mathbf r}_{k-1}}
  {|\partial_s{\mathbf r}_{k}|^2|\partial_s{\mathbf r}_{k-1}|^2}\right) \,ds \right|^2 du \\
& \quad +2\int_0^L \left| \frac{{\mathbf r}_{k}-{\mathbf r}_{k-1}}{\tau}\cdot\partial_s({\mathbf r}_{k}+{\mathbf r}_{k-1})
  \frac{g \,\partial_s{\mathbf r}_{k-1}}{|\partial_s{\mathbf r}_{k}|^2|\partial_s{\mathbf r}_{k-1}|^2} \right|^2 ds\\
& \le M_5 \,\|g\|^2_{\mathcal{H}^1(\Omega_L)}
\end{align*} 
for some $0 < M_4, M_5 < \infty$ (independent of $N \in {\mathbb N}$). Here, again, we used Lemma \ref{le:increase} and the estimates provided in Proposition \ref{stabestir} in several steps.
\end{proof}

\begin{proposition}[Stability bound for Lagrange multiplier]
Let $\lambda_{k+1}\in \mathcal{H}^{-1}(\Omega_L)$ be as in \eqref{eq:lambda}, $k=1,...,N-1$, and let $\lambda^\tau$ be the corresponding linear interpolation. Then $\lambda^\tau\in \mathcal{H}^{-1,0}(\Omega)$ and there exists $0 < C < \infty$, independent of $N \in {\mathbb N}$, such that
\begin{align}\label{eq:boundonlambda}
\big|{}_{\mathcal{H}_{0,L,T}^{1}(\Omega)}{\!\langle g \otimes h, \lambda^{\tau}\rangle}_{\mathcal{H}^{-1}(\Omega)}\big| \le C \,\|g\|_{\mathcal{H}^1(\Omega_L)}\|h\|_{\mathcal{H}^1(\Omega_T)}
\end{align}
for all $g \in \mathcal{H}^{1}_{0,L}(\Omega_L)$, $h \in \mathcal{H}^{1}_{0,T}(\Omega_T)$.
\end{proposition}

\begin{proof}
According to \eqref{eq:lambda}, the Lagrange multiplier is given by
\begin{eqnarray*}
\lambda_{k+1} = -J_{k+1}'({\mathbf r}_{k+1})(e'_{k+1})^{-1}, \quad k = 1, \ldots, N-1.
\end{eqnarray*}
Let $g \in \mathcal{H}^1_{0,L}(\Omega_L)$ and $h \in \mathcal{H}^1_{0,T}(\Omega_T)$. Then, due to the imposed initial and boundary conditions
\begin{align}\label{form1} \nonumber
{}_{\mathcal{H}^{1}_{0,L,T}(\Omega)}&{\!\langle g \otimes h, \lambda^{\tau}\rangle}_{\mathcal{H}^{-1}(\Omega)}
 = \int_0^T {}_{\mathcal{H}^{1}_{0,L}(\Omega_L)}{\!\langle g,\lambda^{\tau}(t)\rangle}_{\mathcal{H}^{-1}(\Omega_L)} h(t) \,dt \\\nonumber
& = \sum_{k=1}^{N-1} \int_{t_k}^{t_{k+1}} \rule[-2ex]{0pt}{0pt}_{{{\mathcal{H}^{1}_{0,L}(\Omega_L)}}}{\!\left\langle g,
  \frac{t-t_{k}}{\tau} ({\lambda}_{k+1}-{\lambda}_{k}) +{\lambda}_{k} \right\rangle}_{\mathcal{H}^{-1}(\Omega_L)} h(t) \,dt \\\nonumber
& = \sum_{k=1}^{N-1} {}_{\mathcal{H}^{1}_{0,L}(\Omega_L)}{\!\langle g,{\lambda}_{k+1} \rangle}_{\mathcal{H}^{-1}(\Omega_L)} h^{(-1)}_{k+1}
 -{}_{\mathcal{H}^{1}_{0,L}(\Omega_L)}{\!\langle g,{\lambda}_{k} \rangle}_{\mathcal{H}^{-1}(\Omega_L)} h^{(-1)}_k \\\nonumber
& \quad - \left( {}_{\mathcal{H}^{1}_{0,L}(\Omega_L)}{\!\langle g,{\lambda}_{k+1} \rangle}_{\mathcal{H}^{-1}(\Omega_L)} 
 -{}_{\mathcal{H}^{1}_{0,L}(\Omega_L)}{\!\langle g,{\lambda}_{k} \rangle}_{\mathcal{H}^{-1}(\Omega_L)} \right) \frac{1}{\tau}\int_{t_k}^{t_{k+1}} h^{(-1)}(t) \,dt \\
& = \tau \sum_{k=2}^{N-1} {}_{\mathcal{H}^{1}_{0,L}(\Omega_L)}{\!\langle g,{\lambda}_{k} \rangle}_{\mathcal{H}^{-1}(\Omega_L)} \big({\rm D}^2h^{(-2)}\big)_{k+1} 
 - {}_{\mathcal{H}^{1}_{0,L}(\Omega_L)}{\!\langle g,{\lambda}_{N} \rangle}_{\mathcal{H}^{-1}(\Omega_L)} \big({\rm D}h^{(-2)}\big)_{N},
\end{align}
where $h^{(-j)}$ is the primitive function of $h^{(-j+1)}$ with $h^{(-j)}(T)=0$, $j = 1,2$, $h^{(0)} = h$.  
Furthermore, 
\begin{align}\label{form2} \nonumber
{}_{\mathcal{H}^{1}_{0,L}(\Omega_L)}{\!\langle g,{\lambda}_{k+1} \rangle}_{\mathcal{H}^{-1}(\Omega_L)} &= -J_{k+1}'({\bf r}_{k+1})(e'_{k+1})^{-1}[g] \\ \nonumber
& = - 2 \Big(\omega \big(({\rm D}^2{\bf r}^\tau)_{k+1}, (e'_{k+1})^{-1}[g] \big)_{\mathcal{L}^2(\Omega_L)} 
  + b \big(\partial_{ss}{\bf r}_{k+1}, \partial_{ss}((e'_{k+1})^{-1}[g]) \big)_{\mathcal{L}^2(\Omega_L)}\\
& \quad \quad \quad - \big({\bf  f}_{k+1}, (e'_{k+1})^{-1}[g] \big)_{\mathcal{L}^2(\Omega_L)}\Big).
\end{align}
Using \eqref{form1} and \eqref{form2}, now we estimate ${}_{\mathcal{H}^{1}_{0,L,T}(\Omega)}{\!\langle g \otimes h,
 \lambda^{\tau}\rangle}_{\mathcal{H}^{-1}(\Omega)}$ term by term. First we consider 
\begin{align*}
\Bigg|\tau \sum_{k=2}^{N-1} \Bigg( \frac{{\bf r}_{k} - 2{\bf r}_{k-1} - {\bf r}_{k-2} }{\tau^2},& (e'_{k})^{-1}[g] \Bigg)_{\mathcal{L}^2(\Omega_L)}
\big({\rm D}^2h^{(-2)}\big)_{k+1} \Bigg| \\
&\le \left|\tau \sum_{k=2}^{N-2} \left( ({\rm D}{\bf r}^\tau)_{k}, \big({\rm D}\big((e')^{-1}\big)^\tau\big)_{k+1}[g] \right)_{\mathcal{L}^2(\Omega_L)} 
  \big({\rm D}^2h^{(-2)}\big)_{k+1} \right| \\
& \quad + \left|\tau \sum_{k=2}^{N-2} \left( ({\rm D}{\bf r}^\tau)_{k}, (e'_{k+1})^{-1}[g] \right)_{\mathcal{L}^2(\Omega_L)} \big({\rm D}^3h^{(-2)}\big)_{k+2} \right| \\
& \quad + \left|\left( ({\rm D}{\bf r}^\tau)_{N-1}, (e'_{N-1})^{-1}[g] \right)_{\mathcal{L}^2(\Omega_L)} \big({\rm D}^2h^{(-2)}\big)_{N} \right| \\
& \le KM \|g\|_{\mathcal{H}^1(\Omega_L)} \,\tau \sum_{k=2}^{N-1} \left(\left|\big({\rm D}^2h^{(-2)}\big)_{k+1} \right| +
 \left|\big({\rm D}^3h^{(-2)}\big)_{k+1} \right| \right) \\
& \le 5\sqrt{T}KM \|g\|_{\mathcal{H}^1(\Omega_L)} \|h\|_{\mathcal{H}^1(\Omega_T)},
\end{align*}
where we used \eqref{eq:boundonr} and \eqref{estinverse}. Since $h^{(-1)}(T)=h(T)=0$, we have
\begin{equation*}
\left|\left( \frac{{\mathbf r}_{N} - 2{\mathbf r}_{N-1} - {\mathbf r}_{N-2} }{\tau^2}, (e'_{N})^{-1}[g] \right)_{\mathcal{L}^2(\Omega_L)}
 \big({\mathrm D}h^{(-2)}\big)_{N} \right| \le C_1 \|g\|_{\mathcal{H}^1(\Omega_L)} \|h\|_{\mathcal{H}^1(\Omega_T)}
\end{equation*}
for some $0 < C_1 < \infty$, independent of $N \in {\mathbb N}$. Using \eqref{eq:boundonr}, \eqref{estinverse}
and the continuity of ${\bf f}$, a derivation of an appropriate bound for the remaining four terms is straight forward.
\end{proof}

%%%%%%%%
\newpage
\subsection{Convergence}

\begin{theorem}\label{th:onconv}
There exists a sequence of discretizations $(\tau_n)_{n \in {\mathbb N}}$, ${\mathbf r} \in \mathcal{H}^{2,1}(\Omega)$ and $\lambda \in \mathcal{H}^{-\beta}(\Omega)$ such that
\begin{subequations}\label{eq:conv}
\begin{align}
\lim_{n\to \infty} {\mathbf r}^{\tau_n} &= {\mathbf r}  \quad \mbox{ in } \quad  \mathcal{C}^{0}([0,T]; \mathcal{L}^2(\Omega_L; {\mathbb R}^3)),\label{eq:conv1} \\
\lim_{n\to \infty} {\mathbf r}^{\tau_n} &= {\mathbf r} \quad \mbox{ weakly in } \quad \mathcal{H}^{2,1}(\Omega; {\mathbb R}^3),\label{eq:conv2}\\
\lim_{n\to \infty} {\mathbf r}^{\tau_n} &= {\mathbf r} \quad \mbox{ strongly  in } \quad \mathcal{L}^2(\Omega; {\mathbb R}^3),\label{eq:conv3} \\
\lim_{n\to \infty} {\lambda}^{\tau_n} &= {\lambda} \quad \mbox{ strongly  in } \quad \mathcal{H}^{-\beta}(\Omega),\label{eq:conv4}
\end{align}
\end{subequations}
for all $3/2 < \beta < \infty$. Furthermore, $({\bf r}, \lambda)$ are weakly solving \eqref{eq:1}, i.e.,
\begin{subequations}\label{eq:existcont}
\begin{align}\label{eq:existcont1}
-\omega \,( \partial_t{\mathbf r}, \partial_t \boldsymbol{\phi} )_{\mathcal{L}^2(\Omega)} 
&= - {}_{\mathcal{H}_{0,L,T}^{1,0}(\Omega)}{\!\langle \partial_s{\mathbf r} \cdot\partial_{s}\boldsymbol{\phi},\lambda \rangle}_{\mathcal{H}^{-1,0}(\Omega)}
- b \,( \partial_{ss}{\mathbf r}, \partial_{ss} \boldsymbol{\phi})_{\mathcal{L}^2(\Omega)}
+({\mathbf  f}, \boldsymbol{\phi})_{\mathcal{L}^2(\Omega)} \\\label{eq:existcont2}
\big(|\partial_s{\mathbf r}|^2,\partial_t \phi \big)_{\mathcal{L}^2(\Omega)} &= 0,
\end{align}
\end{subequations}
for all $\boldsymbol{\phi} \in \mathcal{H}^3_{0,L}(\Omega_L;\mathbb{R}^3) \otimes \mathcal{H}^1_{0,T}(\Omega_T)$ and all ${\phi} \in \mathcal{C}^\infty_c(\Omega)$.
Furthermore, for all $0 \le \gamma <1/2$ 
\begin{align*}
\lim_{n\to \infty} {\bf r}^{\tau_n}(t) &= {\bf r}(t)  \quad \quad \mbox{ in } \quad  \mathcal{C}^{0, \gamma}([0,L]; {\mathbb R}^3), \\
\lim_{n\to \infty} \partial_s{\bf r}^{\tau_n}(t) &= \partial_s{\bf r}(t) \quad \mbox{ in } \quad  \mathcal{C}^{0, \gamma}([0,L]; {\mathbb R}^3),
\end{align*}
for all $t \in D$, where $D \subset [0,T]$ is countable (in the following we are choosing $D \subset [0,T]$ dense and let
$0, T \in {D}$).
Moreover, ${\mathbf r}$ has a (unique) continuous version (denoted by the same symbol) and
\begin{align*}
{\mathbf r}(L,t) = {\mathbf 0}, \quad \partial_s{\mathbf r}(L,t) = -\mathbf{e_g},
\quad {\mathbf r}(s,0) = (L-s)\mathbf{e_g} \quad \mbox{ for all } (s,t) \in [0,L] \times [0,T],
\end{align*}
and even
\begin{align}\label{algconst}
|\partial_s{\mathbf r}(s,t)|^2 = 1 \,\mbox{ for a.e. }\, (s,t) \in [0,L] \times [0,T].
\end{align}
\end{theorem}

\begin{remark}\label{rm:onconv}
The pairing ${}_{\mathcal{H}_{0,L,T}^{1,0}(\Omega)}{\!\langle \partial_s{\mathbf r} \cdot\partial_{s}\boldsymbol{\phi}, \lambda \rangle}_{\mathcal{H}^{-1,0}(\Omega)}$ in \eqref{eq:existcont1} has to be understood in the following sense: $(\partial_s{\bf r}^{\tau_n}(\cdot - \tau_n)\cdot\partial_{s}\boldsymbol{\phi})_{n \in {\mathbb N}}$
is a sequence in $\mathcal{H}_{0,L,T}^{1,0}(\Omega)$ weakly convergent to $\partial_s{\mathbf r} \cdot\partial_{s}\boldsymbol{\phi}$
in $\mathcal{H}_{0,L,T}^{1,0}(\Omega)$,
$({\lambda}^{\tau_n})_{n \in {\mathbb N}}$ is a sequence in $\mathcal{H}_{0,L,T}^{-1,0}(\Omega)$ strongly convergent to ${\lambda}$
in $\mathcal{H}^{-\beta}(\Omega)$ for all $3/2 < \beta < \infty$ and the limit
$\lim_{n \to \infty} {}_{\mathcal{H}_{0,L,T}^{1,0}(\Omega)}{\!\langle \partial_s{\mathbf r}^{\tau_n} \cdot\partial_{s}\boldsymbol{\phi},
\lambda^{\tau_n} \rangle}_{\mathcal{H}^{-1,0}(\Omega)} \in {\mathbb R}$ exists. Hence we set
\begin{equation*}
{}_{\mathcal{H}_{0,L,T}^{1,0}(\Omega)}{\!\langle \partial_s{\mathbf r} \cdot\partial_{s}\boldsymbol{\phi}, \lambda \rangle}_{\mathcal{H}^{-1,0}(\Omega)}
:= \lim_{n \to \infty} {}_{\mathcal{H}_{0,L,T}^{1,0}(\Omega)}{\!\langle \partial_s{\mathbf r}^{\tau_n}(\cdot - \tau_n) \cdot\partial_{s}\boldsymbol{\phi},
\lambda^{\tau_n} \rangle}_{\mathcal{H}^{-1,0}(\Omega)}.
\end{equation*}
\end{remark}

\begin{proof}
From the estimates in \eqref{eq:boundonr} we can conclude that ${\bf r}^{\tau}$ is uniformly (in $N \in {\mathbb N}$) 
Lipschitz continuous in $\mathcal{C}^{0}([0,T]; \mathcal{L}^2(\Omega_L; {\mathbb R}^3))$ and
\begin{align*}
\{{\mathbf r}^{\tau}(t) \,|\, t \in [0,T]\} \subset \{{\mathbf v} \in \mathcal{H}^2(\Omega_L; {\mathbb R}^3) \,|\,
\|\partial_{ss}{\mathbf v}\|_{\mathcal{L}^{2}(\Omega_L)}\le K \},
\end{align*}
which is a relative compact subset of $\mathcal{L}^2(\Omega_L; {\mathbb R}^3)$. Thus, there exists sequence of discretizations $(\tau_n)_{n \in {\mathbb N}}$ and ${\bf r} \in \mathcal{C}^{0}([0,T]; \mathcal{L}^2(\Omega_L; {\mathbb R}^3))$ such that ${\mathbf r}^{\tau_n}$ converges to ${\mathbf r}$ in $\mathcal{C}^{0}([0,T]; \mathcal{L}^2(\Omega_L; {\mathbb R}^3))$ for $n \to \infty$.

The first estimate in \eqref{eq:boundonfullr} gives the existence of a subsequence $(\tau_n)_{n \in {\mathbb N}}$ (denoted the same) and $\tilde{\mathbf r} \in \mathcal{H}^{2,1}(\Omega; {\mathbb R}^3)$ such that ${\mathbf r}^{\tau_n}$ converges weakly to $\tilde{\mathbf r}$ in $\mathcal{H}^{2,1}(\Omega; {\mathbb R}^3)$ for $n \to \infty$. Since
convergence in $\mathcal{C}^{0}([0,T]; \mathcal{L}^2(\Omega_L; {\mathbb R}^3))$ implies strong convergence in $\mathcal{L}^{2}(\Omega; {\mathbb R}^3)$ as well as weak convergence in $\mathcal{H}^{2,1}(\Omega; {\mathbb R}^3)$ implies weak convergence in $\mathcal{L}^{2}(\Omega; {\mathbb R}^3)$, we have $\tilde{\mathbf r}={\mathbf r}$. In particular, this shows \eqref{eq:conv1}-\eqref{eq:conv3}. 

From \eqref{eq:boundonlambda} together with the fact that the embedding $\mathcal{H}^{\beta_1}(\Omega_a) \subset \mathcal{H}^{1}(\Omega_a)$, $a \in \{L,T\}$, is Hilbert--Schmidt
for all $3/2 < \beta_1 < \infty$, we obtain by the kernel theorem, see e.g.~\cite[Chap.~1, \S 2.3]{berezansky:b:1995}, that ${\lambda}^{\tau}$ is uniformly
(in $N \in {\mathbb N}$) bounded in $\mathcal{H}^{-\beta_1}(\Omega)$ for all $3/2 < \beta_1 < \infty$.
Since the embedding $\mathcal{H}^{-\beta_1}(\Omega) \subset \mathcal{H}^{-\beta_1-\beta_2}(\Omega)$ is compact for all $0< \beta_2 < \infty$, there exists a subsequence
$(\tau_n)_{n \in {\mathbb N}}$ and $\lambda \in \mathcal{H}^{-\beta}(\Omega)$ such that ${\lambda}^{\tau_n}$ converges strongly to ${\lambda}$ in $\mathcal{H}^{-\beta}(\Omega)$
as $n \to \infty$ for all $3/2 < \beta < \infty$. 

Multiplying the linear interpolation of \eqref{eq:exist1} with a time-dependent test function and integrating w.r.t.\ time yields
\begin{align}\label{eq:solproof1} \nonumber
\omega \,\big(\big({\rm D}^2 {\bf r}^{\tau}\big)^\tau,\boldsymbol{\phi} \big)_{\mathcal{L}^2(\Omega)} 
&= - {}_{\mathcal{H}_{0,L,T}^1(\Omega)}{\!\langle \partial_{s}{\bf r}^\tau(\cdot-\tau) \cdot\partial_{s}\boldsymbol{\phi},
 \lambda^\tau\rangle}_{\mathcal{H}^{-1}(\Omega)} \\ 
& \quad - b \,( \partial_{ss}{\bf r}^\tau, \partial_{ss} \boldsymbol{\phi})_{\mathcal{L}^2(\Omega)}+({\bf  f}^\tau, \boldsymbol{\phi})_{\mathcal{L}^2(\Omega)},
\end{align}
for all $\boldsymbol{\phi} \in \mathcal{H}^2_{0,L}(\Omega_L;\mathbb{R}^3) \otimes \mathcal{H}^1_{0,T}(\Omega_T)$
(because $\lambda_0 = \lambda_1 = 0$, on $[-\tau, 0]$ we can assign to the function $\partial_{s}{\mathbf r}^\tau(\cdot-\tau)$ any value, for simplicity we choose zero). Since $\partial_{ss}{\mathbf r}^{\tau_n}$ converges weakly to $\partial_{ss}{\mathbf r}$ and ${\mathbf f}^{\tau_n}$ converges strongly to ${\mathbf f}$, both in $\mathcal{L}^2(\Omega; {\mathbb R}^3)$, as $n \to \infty$, we have
\begin{equation}\label{eq:solproof2}
\lim_{n \to \infty} ( \partial_{ss}{\mathbf r}^{\tau_{n}}, \partial_{ss} \boldsymbol{\phi})_{\mathcal{L}^2(\Omega)} = ( \partial_{ss}{\mathbf r}, \partial_{ss} \boldsymbol{\phi})_{\mathcal{L}^2(\Omega)}, \qquad \lim_{n \to \infty} ({\mathbf  f}^{\tau_n}, \boldsymbol{\phi})_{\mathcal{L}^2(\Omega)}
  = ({\mathbf  f}, \boldsymbol{\phi})_{\mathcal{L}^2(\Omega)},
\end{equation}
for all $\boldsymbol{\phi} \in \mathcal{H}^2_{0,L}(\Omega_L;\mathbb{R}^3) \otimes \mathcal{H}^1_{0,T}(\Omega_T)$.
Furthermore, integration by parts yields
\begin{align}\label{eq:solproof3} \nonumber
\big(\big({\rm D}^2 & {\bf r}^{\tau}\big)^\tau,\boldsymbol{\phi} \big)_{\mathcal{L}^2(\Omega)} 
= \sum_{k=0}^{N-1} \int_{t_k}^{t_{k+1}} \int_0^L \bigg(\frac{t-t_k}{\tau}
  \Big(({\rm D}^2{\bf r}^\tau)_{k+1} - ({\rm D}^2{\bf r}^\tau)_{k}\Big) + ({\rm D}^2{\bf r}^\tau)_{k} \bigg) \cdot \boldsymbol{\phi} \,ds\,dt \\\nonumber
& = \sum_{k=0}^{N-1} \int_{t_k}^{t_{k+1}} \int_0^L \frac{1}{\tau}\bigg(\bigg(\frac{t-t_k}{\tau}
  \Big(({\rm D}{\bf r}^\tau)_{k+1} - ({\rm D}{\bf r}^\tau)_{k}\Big) + ({\rm D}{\bf r}^\tau)_{k} \bigg) \\\nonumber
& \quad -\bigg(\frac{t-t_k}{\tau} \Big(({\rm D}{\bf r}^\tau)_{k} - ({\rm D}{\bf r}^\tau)_{k-1}\Big) + ({\rm D}{\bf r}^\tau)_{k-1} \bigg)\bigg)
  \cdot \boldsymbol{\phi} \,ds\,dt \\\nonumber
& = - \sum_{k=0}^{N-2} \int_{t_k}^{t_{k+1}} \int_0^L \bigg(\frac{t-t_k}{\tau}
  \Big(({\rm D}{\bf r}^\tau)_{k+1} - ({\rm D}{\bf r}^\tau)_{k}\Big) + ({\rm D}{\bf r}^\tau)_{k} \bigg)
  \cdot \frac{\boldsymbol{\phi}(\cdot + \tau) - \boldsymbol{\phi}}{\tau} \,ds\,dt \\\nonumber
& \quad + \int_{t_{N-1}}^{t_N} \int_0^L \frac{1}{\tau}\bigg(\frac{t-t_{N-1}}{\tau}
  \Big(({\rm D}{\bf r}^\tau)_{N} - ({\rm D}{\bf r}^\tau)_{N-1}\Big) + ({\rm D}{\bf r}^\tau)_{N-1} \bigg)
  \cdot \boldsymbol{\phi} \,ds\,dt \\\nonumber
& = \int_0^L \Big(({\rm D}{\bf r}^\tau)_{N} - ({\rm D}{\bf r}^\tau)_{N-1}\Big)
  \cdot \frac{1}{\tau^2} \int_{t_{N-1}}^{t_N} \int_\cdot^{t_N} \boldsymbol{\phi} \,du\,dt\,ds
  + \int_0^L ({\rm D}{\bf r}^\tau)_{N-1} \cdot \frac{1}{\tau} \int_{t_{N-1}}^{t_N} \boldsymbol{\phi} \,dt\,ds \\
& \quad - \int_{t_1}^{t_{N-1}} \int_0^L \big({\rm D} {\bf r}^{\tau}\big)^\tau
  \cdot \frac{\boldsymbol{\phi}(\cdot + \tau) - \boldsymbol{\phi}}{\tau} \,ds\,dt.
\end{align}
By \eqref{eq:boundonr} together with the boundary conditions imposed on $\boldsymbol{\phi}$ now from \eqref{eq:solproof3} it follows
\begin{equation}\label{eq:solproof4}
\lim_{n \to \infty} \big(\big({\rm D}^2 {\bf r}^{\tau_n}\big)^{\tau_n},\boldsymbol{\phi} \big)_{\mathcal{L}^2(\Omega)} 
= - \lim_{n \to \infty} \int_{t_1}^{t_{N-1}} \int_0^L \big({\rm D} {\bf r}^{\tau_n}\big)^{\tau_n}
\cdot \frac{\boldsymbol{\phi}(\cdot + \tau_n) - \boldsymbol{\phi}}{\tau_n} \,ds\,dt
= - \big(\partial_t{\bf r},\partial_t\boldsymbol{\phi} \big)_{\mathcal{L}^2(\Omega)},
\end{equation}
where in the last step we used that $\partial_t {\bf r}^{\tau_n}$ converges weakly to 
$\partial_t {\bf r}$, $(\boldsymbol{\phi}(\cdot + \tau_n) - \boldsymbol{\phi})/\tau_n$ converges strongly to $\partial_t \boldsymbol{\phi}$ and 
$\partial_t {\bf r}^{\tau_n} - \big({\rm D} {\bf r}^{\tau_n}\big)^{\tau_n}$ converges strongly to $0$, all in $\mathcal{L}^2(\Omega; {\mathbb R}^3)$ as $n \to \infty$.

Combining \eqref{eq:solproof1}, \eqref{eq:solproof2}, \eqref{eq:solproof4} we obtain
\begin{align}\label{limitexist} \nonumber
\lim_{n \to \infty} {}_{\mathcal{H}_{0,L,T}^{1,0}(\Omega)}{\!\langle \partial_{s}{\bf r}^{\tau_n}(\cdot-\tau_n)
\cdot\partial_{s}\boldsymbol{\phi}, \lambda^{\tau_n}\rangle}_{\mathcal{H}^{-1,0}(\Omega)} 
& = \omega \,\big(\partial_t{\bf r},\boldsymbol{\phi} \big)_{\mathcal{L}^2(\Omega)} 
  - b \,( \partial_{ss}{\bf r}, \partial_{ss} \boldsymbol{\phi})_{\mathcal{L}^2(\Omega)}\\
& \quad  +({\bf  f}, \boldsymbol{\phi})_{\mathcal{L}^2(\Omega)},
\end{align}
for all $\boldsymbol{\phi} \in \mathcal{H}^2_{0,L}(\Omega_L;\mathbb{R}^3) \otimes \mathcal{H}^1_{0,T}(\Omega_T)$.
Now we restrict ourself to $\boldsymbol{\phi} \in \mathcal{H}^3_{0,L}(\Omega_L;\mathbb{R}^3) \otimes \mathcal{H}^1_{0,T}(\Omega_T)$. Since $\partial_{s}{\bf r}^{\tau_n}$ converges weakly to $\partial_{s}{\mathbf r}$ in $\mathcal{H}_{0,L,T}^{1,0}(\Omega; {\mathbb R}^3)$ and $\partial_s\boldsymbol{\phi}, \partial_{ss}\boldsymbol{\phi}$ are
bounded functions, also $\partial_{s}{\mathbf r}^{\tau_n}(\cdot-\tau_n) \cdot\partial_{s}\boldsymbol{\phi}$ converges weakly to $\partial_{s}{\mathbf r} \cdot\partial_{s}\boldsymbol{\phi}$ in $\mathcal{H}_{0,L,T}^{1,0}(\Omega)$ as $n \to \infty$. Furthermore, $\lambda^{\tau_n}$ converges strongly to $\lambda$
in $\mathcal{H}^{-\beta}(\Omega)$ as $n \to \infty$ for all $3/2 < \beta < \infty$. Thus we identify 
\begin{equation}\label{limitright}
\lim_{n \to \infty} {}_{\mathcal{H}_{0,L,T}^{1,0}(\Omega)}{\!\langle \partial_{s}{\mathbf r}^{\tau_n}(\cdot-\tau_n)
\cdot\partial_{s}\boldsymbol{\phi}, \lambda^{\tau_n}\rangle}_{\mathcal{H}^{-1,0}(\Omega)}
= {}_{\mathcal{H}_{0,L,T}^{1,0}(\Omega)}{\!\langle \partial_{s}{\mathbf r} \cdot\partial_{s}\boldsymbol{\phi},
\lambda \rangle}_{\mathcal{H}^{-1,0}(\Omega)},
\end{equation}
for all $\boldsymbol{\phi} \in \mathcal{H}^3_{0,L}(\Omega_L;\mathbb{R}^3) \otimes \mathcal{H}^1_{0,T}(\Omega_T)$ in the sense of Remark \ref{rm:onconv}.
Hence, \eqref{eq:existcont1} follows from \eqref{limitexist} together with \eqref{limitright}.

Now, using \eqref{eq:exist2}, we obtain for all ${\phi} \in \mathcal{C}^\infty_c(\Omega)$
\begin{align}\label{preweakbd} \nonumber
\big|\big(|\partial_s{\bf r}^{\tau}|^2,\partial_t \phi \big)_{\mathcal{L}^2(\Omega)}\big|
&= \bigg|\sum_{k=0}^{N-1} \int_{t_k}^{t_{k+1}} \int_0^L \bigg( \frac{t-t_k}{\tau}
\Big( \partial_s{\bf r}_{k+1} - \partial_s{\bf r}_{k} \Big) + \partial_s{\bf r}_{k} \bigg)^2 \partial_t \phi \,ds\,dt \bigg|\\\nonumber
&= \bigg|\frac{2}{\tau} \sum_{k=0}^{N-1} \int_{t_k}^{t_{k+1}} \int_0^L \bigg( \frac{t-t_k}{\tau}
\Big( \partial_s{\bf r}_{k+1} - \partial_s{\bf r}_{k} \Big) + \partial_s{\bf r}_{k} \bigg)
\Big( \partial_s{\bf r}_{k+1} - \partial_s{\bf r}_{k} \Big) \phi \,ds\,dt \bigg|\\\nonumber
&= 2 \sum_{k=0}^{N-1} \int_{t_k}^{t_{k+1}} \int_0^L \frac{t-t_k}{\tau^2}
\Big( |\partial_s{\bf r}_{k+1}|^2 - |\partial_s{\bf r}_{k}|^2 \Big) |\phi| \,ds\,dt \\\nonumber
&\le \|\phi\|_{\mathcal{C}^0(\overline{\Omega})} \sum_{k=0}^{N-1} \int_{0}^{L} \Big( |\partial_s{\bf r}_{k+1}|^2 - |\partial_s{\bf r}_{k}|^2 \Big) \,ds \\
& = \|\phi\|_{\mathcal{C}^0(\overline{\Omega})} \int_{0}^{L} \Big( |\partial_s{\bf r}_{N}|^2 - 1|^2 \Big) \,ds
\le R \|\phi\|_{\mathcal{C}^0(\overline{\Omega})} \sqrt{\tau}
\end{align}
due to Lemma \ref{le:increase} and Proposition \ref{estimateconstraint}. Because ${\bf r}^{\tau_n}$ converges strongly to ${\mathbf r}$ and $\partial_s{\mathbf r}^{\tau_n}$, $\partial_{ss}{\mathbf r}^{\tau_n}$ converge weakly to $\partial_s{\mathbf r}$, $\partial_{ss}{\mathbf r}$, respectively, in $\mathcal{L}^2(\Omega; {\mathbb R}^3)$ as $n \to \infty$, by an integration by parts together with \eqref{preweakbd} we can conclude
\begin{align}\label{weakzero} \nonumber
\big(|\partial_s{\bf r}|^2,\partial_t \phi \big)_{\mathcal{L}^2(\Omega)}
&= - \big({\bf r} \cdot \partial_s{\bf r},\partial_{st} \phi \big)_{\mathcal{L}^2(\Omega)} 
- \big({\bf r} \cdot \partial_{ss}{\bf r}, \partial_t \phi \big)_{\mathcal{L}^2(\Omega)}\\\nonumber
&= - \lim_{n \to \infty}\big(\partial_{st} \phi \,\, {\bf r}^{\tau_n}, \partial_s{\bf r}^{\tau_n} \big)_{\mathcal{L}^2(\Omega)} 
- \lim_{n \to \infty}\big(\partial_{t} \phi \,\, {\bf r}^{\tau_n}, \partial_{ss}{\bf r}^{\tau_n} \big)_{\mathcal{L}^2(\Omega)}\\
&= \lim_{n \to \infty}\big(|\partial_s{\bf r}^{\tau_n}|^2,\partial_t \phi \big)_{\mathcal{L}^2(\Omega)} = 0
\end{align}
for all ${\phi} \in \mathcal{C}^\infty_c(\Omega)$, i.e., \eqref{eq:existcont2} is shown.

Due to the second estimate in \eqref{eq:boundonr}, for each $t \in [0,T]$ there exists a subsequence $(\tau_n)_{n \in {\mathbb N}}$ (depending on $t$) such that
\begin{equation}\label{eq:pointwise}
\lim_{n\to \infty} {\mathbf r}^{\tau_n}(t) = {\mathbf r}(t), \qquad
\lim_{n\to \infty} \partial_s{\mathbf r}^{\tau_n}(t) = \partial_s{\mathbf r}(t) \qquad \mbox{ both in } \mathcal{C}^{0, \gamma}([0,L]; {\mathbb R}^3),
\end{equation}
for all $0 \le \gamma < 1/2$. Let $D \subset [0,T]$ be countable. Then, by dropping to subsequences and taking the diagonal sequence, we obtain \eqref{eq:pointwise} for all $t \in D$. Here we choose $D \subset [0,T]$ dense with $0, T \in [0,T]$. From this, together with the estimates in \eqref{eq:boundonr}, we can conclude that ${\bf r}$ has a (unique) continuous version on $[0,L] \times [0,T]$ (which we denote by the same symbol). Moreover,
\begin{align}\label{tzero}
{\mathbf r}(L,t) = {\mathbf 0}, \qquad \partial_s{\mathbf r}(L,t) = -\mathbf{e_g},
\qquad {\mathbf r}(s,0) = (L-s) \mathbf{e_g} \qquad \mbox{ for all } (s,t) \in [0,L] \times [0,T].
\end{align}
Finally, \eqref{weakzero} together with \eqref{tzero} implies \eqref{algconst}.  
\end{proof}

%%%%%%%%%%%%%%%%%%%%%%%%%%%%%%%%%%%%%%%%%%%%%%%%%%%%%%%%%%%%
\setcounter{equation}{0} 
\setcounter{figure}{0} 
\section{Numerical study}\label{sec:4}

In this section we present exemplary simulations of elastic fiber motions. The numerical results regarding convergence, fiber elongation and longtime behavior coincide well with the previous analytical investigations.

%%%%%%%%%%
\subsection{Spatial finite element discretization}
On every time level $t_{k+1}$, the semi-discretized fiber system \eqref{eq:2} corresponds to a constrained minimization problem in a Hilbert space setting. We solve the associated adjoint problem (linear saddle point problem \eqref{eq:ad}) in a finite dimensional approximation space by choosing finite element spaces of piecewise cubic polynomials for the curve and Dirac distributions for the Lagrange multiplier. To facilitate the readability we suppress here the actual time index and indicate quantities associated to the spatial discretization by the subindex $_h$.

We use conforming finite element spaces $\mathbb{H}_h\subset \mathcal{H}^2(\Omega_L;\mathbb{R}^n)$ that are subordinated to a partition of $\overline{\Omega}_L=[0,L]$ into subintervals of length $h$. The partition is identified with the sequence of nodes $s_j=jh$, $j=0,...,M$, $M=L/h$. Certainly one can also think of different $h_j$, then the partition is assumed to satisfy $h=\max_{j} h_j\rightarrow 0$ as $M\rightarrow \infty$.  We span $\mathbb{H}_h$ by a node basis of cubic splines
\begin{align*}
\psi_j(s)&=\left\{ \begin{array}{l l l} 
		  (2+3x-x^3)/4, \hspace*{1.1cm} & x=(2s-(s_j+s_{j-1}))/h, & s\in [s_{j-1},s_j]\\
                  (2-3x+x^3)/4, & x=(2s-(s_{j+1}+s_{j}))/h, & s\in [s_{j},s_{j+1}]\\
		  0 & & else
                  \end{array}
	  \right.\\
\varphi_j(s)&=\left\{ \begin{array}{l l l}
		  h(-1-x+x^2+x^3)/8, & x=(2s-(s_j+s_{j-1}))/h, & s\in [s_{j-1},s_j]\\
                  h(\,\,\,1-x-x^2+x^3)/8, & x=(2s-(s_{j+1}+s_{j}))/h, & s\in [s_{j},s_{j+1}]\\
		  0 & & else
                  \end{array}
	  \right.,
\end{align*}
where we consider $s_{-1}=s_0$ and $s_{M+1}=s_M$ to simplify the notation. Then, any function $\mathbf{v}_h\in \mathbb{H}_h$ \begin{align*} 
\mathbf{v}_h(s)=\sum_{j=0}^M \mathsf{v}^\circ_j \psi_j(s) + \mathsf{v}'_j \varphi_j(s), 
\qquad \mathbf{v}_h \in \mathcal{C}^1(\Omega_L;\mathbb{R}^n)
\end{align*}
is represented by its coefficient tuple $\mathsf{v}=(\mathsf{v}^\circ_0,...\mathsf{v}^\circ_M,\mathsf{v}'_0,...\mathsf{v}'_M)^T\in \mathbb{R}^{2(M+1)n}$. In particular, $\mathsf{v}^\circ_j$ and $\mathsf{v}'_j\in \mathbb{R}^n$ describe the values of the function and its derivative at the node $s_j$, since $\psi_j(s_i)=\delta_{ij}$, $\partial_s \psi_j(s_i)=0$ and $\varphi_j(s_i)=0$, $\partial_s\varphi_j(s_i)=\delta_{ij}$ hold true. In the finite dimensional fiber space $\mathbb{V}_h\subset \mathbb{H}_h$, the degrees of freedom reduce to $2Mn$ because of the Dirichlet boundary conditions posed at $s=L$ that fix the coefficients $\mathsf{v}^\circ_M$ and $\mathsf{v}'_M$.  Piecewise polynomial functions cannot fulfill the arc-length constraint in the whole $\Omega_L$, unless they are globally affine. To allow for a fiber dynamics $\mathbf{v}_h\neq(\mathbf{r}_0)_h$ over time, we introduce a finite dimensional basis of Dirac distributions, i.e.\ $\eta_i(s)=\delta(s-\hat{s}_i)$, $i=1,...,\hat{M}$, for the approximation of the Lagrange multiplier $\lambda_h \in \mathbb{\hat H}_h \subset \mathcal{H}^{-1}(\Omega_L)$ and satisfy the constraint only at the respective points $\hat{s}_i$. These constraint points $\hat{s}_i$ are located with respect to the underlying partition. The total number of constraint points depends on the degrees of freedom and is a compromise between approximation quality and numerical realization, we set $\hat{M}=QM$ for a uniform distribution, $Q \in \mathbb{N}$. The intuitive choice are certainly the nodes (cf.\ \cite{bartels:p:2013}), yielding $(\mathsf{v}_j'-(\mathsf{r}_k)_j')\cdot (\mathsf{r}_k)_j'=0$, $j=0,...,M-1$, with $\mathsf{r}_k$ coefficient tuple associated with $(\mathbf{r}_k)_h$, here $Q=1$. But the constraint can be also imposed more than once per subinterval, for example at the nodes $s_j$ and the cell midpoints $s_j+h/2$, $j=0,...,M-1$ for $Q=2$. In the following we refer to these two variants as $Q=1$ and $Q=2$.

Given $\mathsf{r}_0$ and $\mathsf{r}_1$, the numerical scheme for the fiber dynamics requires then the sequential solving of linear systems of equations in $\mathbb{R}^{2Mn+\hat{M}}$ 
\begin{align}\label{eq:system}
\left( \begin{array}{c c} \Phi + {\tau^2}/{\mu^2} \,\Phi'' & \mathsf{B}^T_k \\ \mathsf{B}_k & 0 \end{array}\right) 
\left( \begin{array}{c}  \mathsf{v} \\ \mathsf{\lambda}\end{array}\right) = \left( \begin{array}{c c} \Phi (2\mathsf{r}_k-\mathsf{r}_{k-1}) + \mathsf{f} \\\mathsf{B}_k \mathsf{r}_k \end{array}\right), \quad \mathsf{r}_{k+1}=\mathsf{v}, \quad k=1,...,N-1 
\end{align} 
with the time-independent symmetric mass $\Phi$ and stiffness matrices $\Phi''$ that are associated with the spline basis, i.e.\ $(\Phi)_{pq}=\int \phi_q \,\phi_p \, ds$ and $(\Phi'')_{pq}=\int \partial_{ss} \phi_q\, \partial_{ss}\phi_p \, ds$, $\phi \in \{ \psi, \varphi \} $. The matrix $\mathsf{B}_k$ corresponds to the constraint conditions. The acting outer forces and the Dirichlet boundary conditions are incorporated in $\mathsf{f}$. In the stated dimensionless form that results from scaling with the fiber length $L$ and a typical velocity $V$, the ratio between inertial and bending effects $\mu=L/V\sqrt{\omega/b}$ characterizes the fiber behavior. The numerical realization is performed with MATLAB, Version R2014a, using the direct solvers.

%%%%%%%%%%
\subsection{Results and discussion}

As benchmark we consider the dynamics of a cantilever beam under gravity, cf.~\cite{baus:p:2015}. The set-up in the dimensionless form is particularly given by $\mathbf{e_g}=\mathbf{e_1}$, $\mathbf{f}=-\mathbf{e_3}/\mathrm{Fr}^2$,  $\mathrm{Fr}=1$ and $\mu=10$ with $\{\mathbf{e_1},\mathbf{e_2},\mathbf{e_3}\}$ Cartesian basis in $\mathbb{R}^3$ and $L=1$ due to the scaling. The dimensionless Froude number $\mathrm{Fr}$ represents the ratio of inertial and gravitational forces. Figure~\ref{fig:cantilever}(left) illustrates the fiber dynamics over time $[0,T]$ for $T=2.5$, for this purpose the fiber curve is illustrated at depicted time points. The computation is performed with $\tau=h=2\cdot10^{-2}$, but even much coarser discretizations yield the same qualitative behavior. The fiber elongation $\Delta L^\tau(t)=\int |\partial_s \mathbf{r}^\tau(t)|ds-L\geq 0$ that is originated in Lemma~\ref{le:increase} reduces for smaller time steps,  $\Delta L^\tau  \rightarrow 0$ for $\tau \rightarrow 0$. For the clamped boundary conditions we observe $\Delta L^\tau \sim \mathcal{O}(\tau)$ in consistence to the investigations in \cite{bartels:p:2013}. In contrast to a non-inertial frictional elastic flow (first order in time) where the elongation is bounded by the initial conditions \cite{bartels:p:2013}, the error bound (Proposition~\ref{estimateconstraint}) depends here crucially on the acting forces $\|\mathbf{f}\|_{\mathcal{L}^2(\Omega)}$ and the end time $T$. Figure~\ref{fig:cantilever}(right) shows the respective longtime behavior of $\Delta L_h^\tau(t)$, $t\in[0,2.5]$ for fixed $h=2\cdot 10^{-2}$ and varying $\tau$. The occurring integrals over $\Omega_L$ are evaluated on basis of the finite element basis by help of a Simpson quadrature rule (error tolerance $\mathrm{tol}=10^{-12}$). 

\begin{figure}[t]
\includegraphics[width=0.475\textwidth]{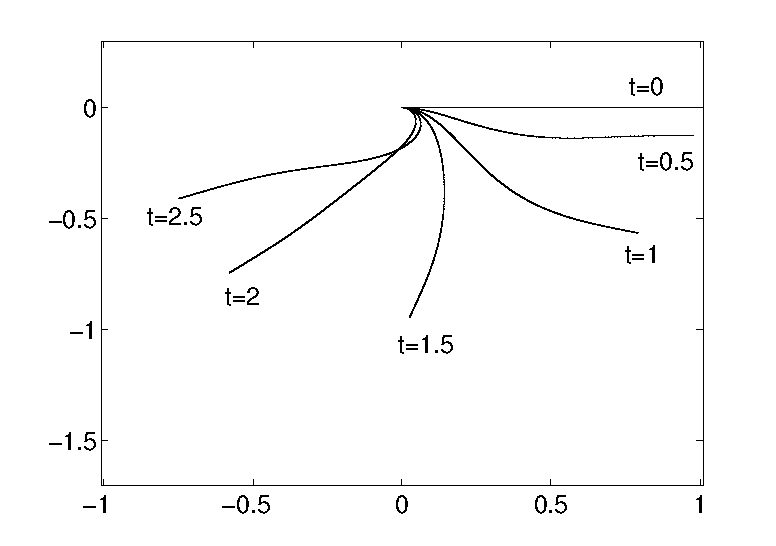}\hfill 
\includegraphics[width=0.475\textwidth]{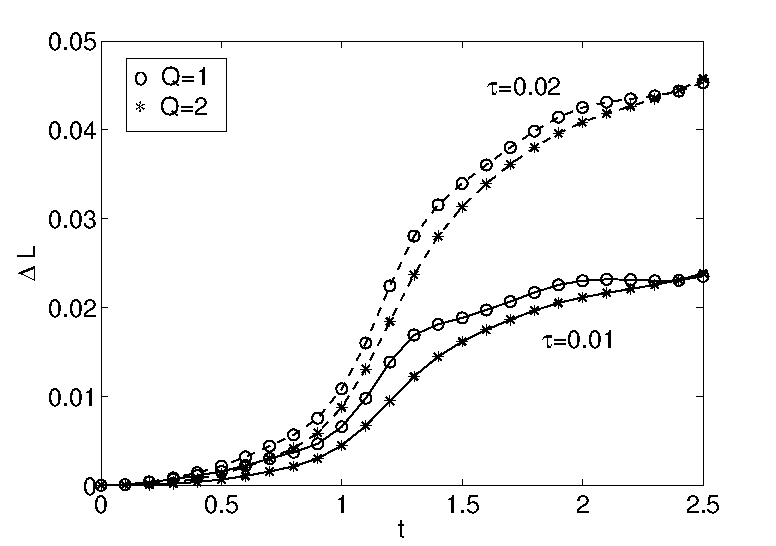}\\
\caption{\label{fig:cantilever} Benchmark test. \emph{Left:} Dynamics of a cantilever beam under gravity. ($\tau=h=0.02$, $Q=1$). \emph{Right:} Elongation over time for $h=0.02$ and varying $\tau$, $Q$.}
\end{figure}

\begin{figure}[b]
\includegraphics[width=0.475\textwidth]{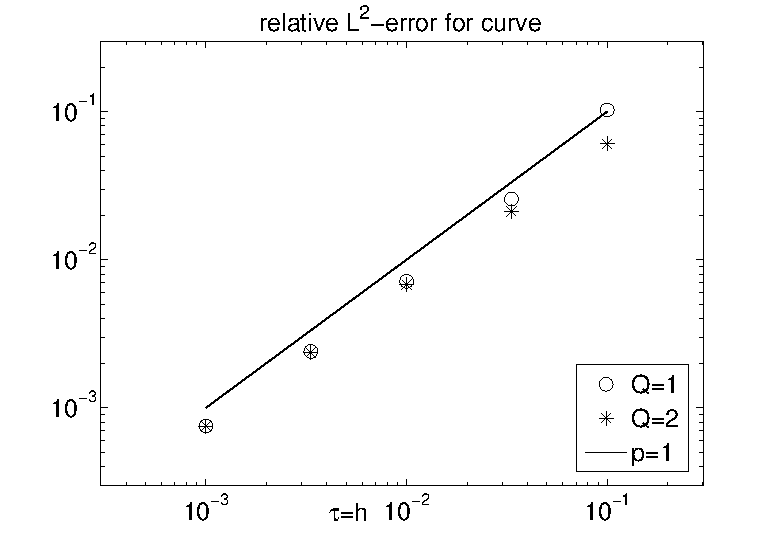}\hfill 
\includegraphics[width=0.475\textwidth]{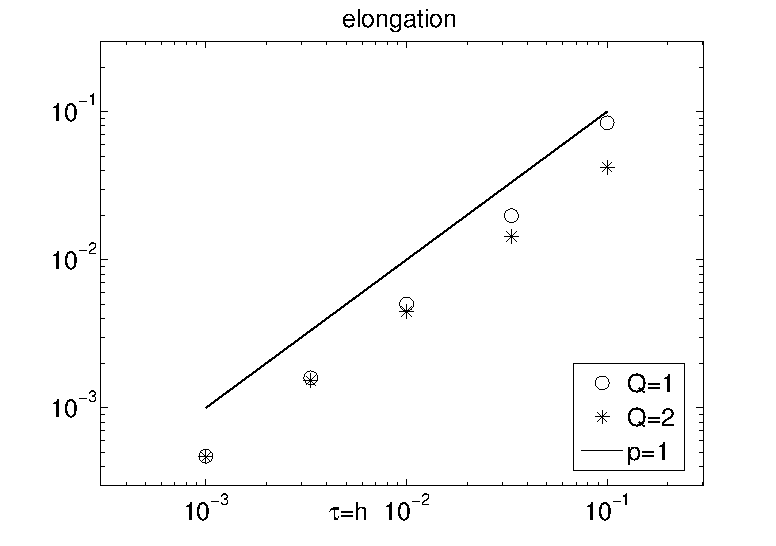}\\
\includegraphics[width=0.475\textwidth]{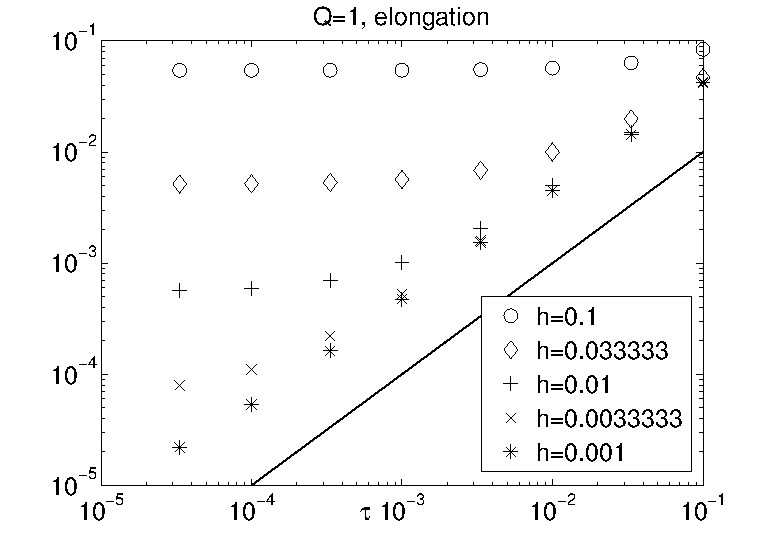}\hfill
\includegraphics[width=0.475\textwidth]{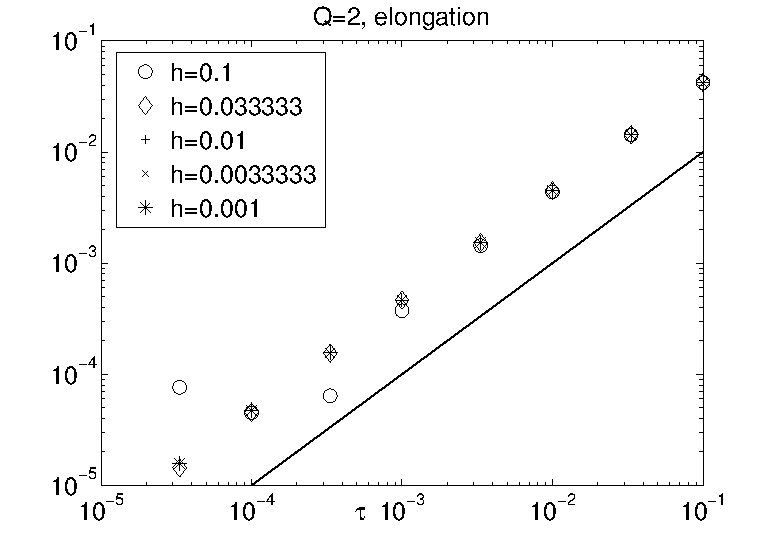}
\caption{\label{fig:cantilever_conv} Benchmark test. \emph{Top:} Space-time convergence for fiber position \emph{(left)} and elongation \emph{(right)} at $t=1$ for $Q=1,2$ as $\tau=h\rightarrow 0$. \emph{Bottom:}  Elongation $\Delta L_h^\tau(t=1)$ for different fixed $h$ for $Q=1$ (\emph{left}) and  $Q=2$ (\emph{right}) as $\tau\rightarrow 0$. The solid line indicates the convergence order $p=1$.}
\end{figure}

The convergence results that are visualized in Fig.~\ref{fig:cantilever_conv} refer to the exemplary time point $t=1$.  The numerical convergence rate of first order in space-time confirms the theory: the relative $\mathcal{L}^2(\Omega_L)$-error for the fiber position $\|(\mathbf{r}_{h^\star}^{\tau^\star}-\mathbf{r}_h^\tau)/\mathbf{r}_{h^\star}^{\tau^\star}\,(t)\|_{\mathcal{L}^2(\Omega_L)}$ is linear as $\tau=h\rightarrow 0$, as reference solution we use here an approximation associated with a sufficiently fine discretization, $(\tau^\star,h^\star)=\overline{3}\cdot(10^{-6},10^{-4})$. The same is found for the elongation $\Delta L_h^\tau$, Fig.~\ref{fig:cantilever_conv}(top). The actual magnitude of the deviation is affected by the finite dimensional approximation of the constraint. It turns out that imposing the constraint not only at the nodes ($Q=1$) but also at the cell midpoints ($Q=2$) yields a much better length preservation for coarser discretizations. Figure~\ref{fig:cantilever_conv}(bottom) shows the influence of $Q=1$ and $Q=2$ on $\Delta L_h^\tau(t=1)$ for different fixed spatial discretizations $h$ as $\tau\rightarrow 0$. For $Q=1$ we clearly see the linear decay that turns into a constant as $\tau\rightarrow0$, these constants depend on $h$ and represent the respective spatial errors. For $Q=2$ the spatial errors are much smaller. For example, $\Delta L_{h}^\tau \sim \mathcal{O}(10^{-4})$, $\tau \rightarrow 0$ requires only $h=10^{-1}$ for $Q=2$ in contrast to $h=\overline{3}\cdot 10^{-3}$ for $Q=1$. This accuracy goes with smaller linear systems \eqref{eq:system} (in $\mathbb{R}^{80}$ for $Q=2$ versus $\mathbb{R}^{21000}$ for $Q=1$ wrt.\ $n=3$) and hence with significant less computational effort.

We conclude the numerical experiments with the simulation of a cantilever beam under gravity and an additional time-dependent periodic transversal force which causes a fully three-dimensional motion, i.e.\ $\mathbf{e_g}=\mathbf{e_1}$, $\mathbf{f}(t)=-\mathbf{e_3}/\mathrm{Fr}^2-\mathbf{e_2} \cos(2\pi t)/\mathrm{Dr}^2$,  $\mathrm{Fr}=1$. The dimensionless parameter $\mathrm{Dr}$ represents the ratio of the inertial and transversal outer forces. Figure~\ref{fig:dynamics} illustrates the effect of $\mu$ and $\mathrm{Dr}$ on the fiber behavior: larger $\mu$ imply a smaller bending stiffness and hence more curvature, smaller $\mathrm{Dr}$ yield more pronounced oscillations out off the $\mathbf{e_1}$-$\mathbf{e_3}$-plane. The respective computations for $t\in [0,T]$, $T=2.5$ are performed with $\tau=10^{-3}$, $h=2\cdot 10^{-2}$ and $Q=2$, the elongation satisfies $\Delta L_h^\tau(T)\leq 10^{-2}$ in all cases. 
\begin{figure}[t]
\hspace*{-0.5cm}
\includegraphics[width=0.365\textwidth]{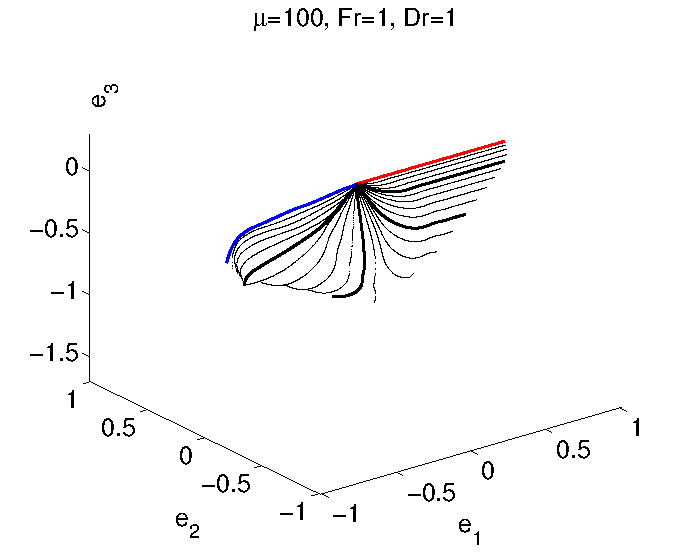} \hspace*{-0.8cm}
\includegraphics[width=0.365\textwidth]{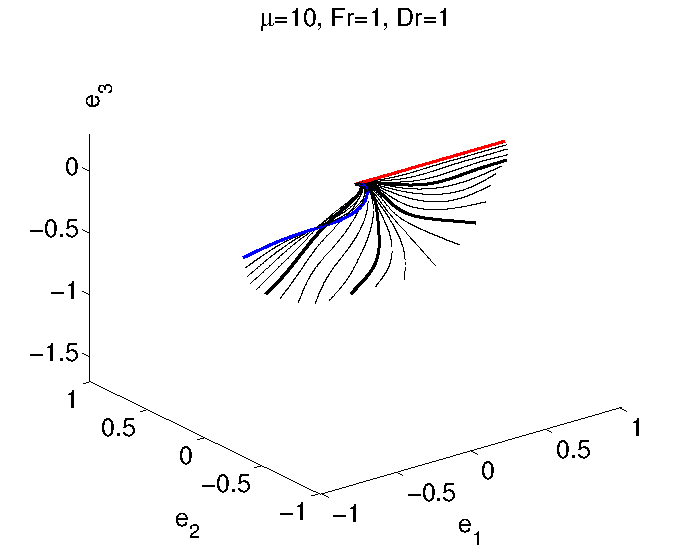} \hspace*{-0.8cm}
\includegraphics[width=0.365\textwidth]{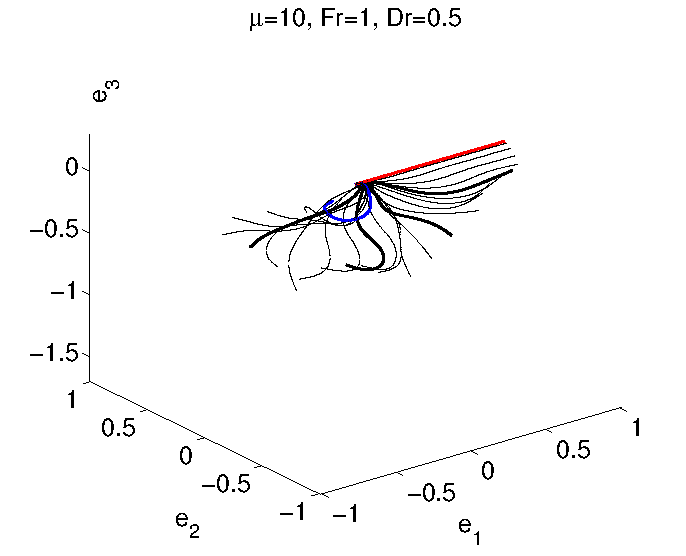}
\caption{\label{fig:dynamics} Dynamics of a cantilever beam under gravity and a periodic transversal force with $\mathrm{Fr}=1$ and varying $(\mu, \mathrm{Dr})$, $t\in [0, 2.5]$. The initial fiber position is visualized by the red line, the final one by the blue line. The remaining lines indicate the motion wrt.\ time steps of size $\Delta t=0.1$. See Fig.~\ref{fig:cantilever}(top, left) for $\mathrm{Dr}\rightarrow \infty$. ($\tau=0.001$, $h=0.02$, $Q=2$).}
\end{figure}

Note that in the implementation the numerical scheme can be easily extended to cover also fiber motions that are driven by curve-dependent outer forces $\mathbf{f}[\mathbf{r},\partial_t\mathbf{r},\partial_s\mathbf{r},s,t]$. When dealing with non-linear forces, it is advantageous to incorporate the linearized constraint in the used fixed point iteration (e.g.\ Newton method) since it improves the accuracy while the expenses are neutral.

%%%%%%%%%%%%%%%%%%%%%%%%%%%%%%%%%%%%%%%%%%%%%%%%%%%%%%%%%%%
\section{Conclusion}

In the technical textile industry the dynamics of an elastic inextensible inertial fiber is modeled by a wavelike, nonlinear fourth order partial differential algebraic system. In this paper we proposed a numerical scheme focusing on the efficient and accurate treatment of the constraint for the local length preservation. A convergence proof and an explicit error bound were presented. Ongoing work deals with the extension of analysis and numerics to the stochastic partial differential algebraic system \cite{marheineke:p:2011} arising for fibers immersed in turbulent air flows. Here, a stochastic force (source term) of a white noise type is added in the model system. The challenge lies again in the handling of the constraint. So far, the corresponding extensible beam equations with additive Gaussian noise have been studied in \cite{baur:p:2013}.

\quad\\
{\sc Acknowledgments} This work has been supported by Bundesministerium f\"ur Bildung und For\-schung, Schwerpunkt "Mathematik f\"ur Innovationen in Industrie and Dienstleistungen", Projekt 05M10, 05M13. We are grateful to an unknown referee
for valuable references from which we got important ideas to treat the algebraic constraint.

%%%%%%%%%%%%%%%%%%%%%%%%%%%%%%%%%%%%%%%%%%%%%%%%%%%%%%%%%%%
\bibliographystyle{siam}
\bibliography{ref_2013}

\begin{thebibliography}{10}

\bibitem{adams:b:1990}
{\sc R.~A. Adams}, {\em Sobolev Spaces}, Academic Press, Boston, 1990.

\bibitem{antman:b:2006}
{\sc S.~S. Antman}, {\em Nonlinear Problems of Elasticity}, Springer, New York,
  2006.

\bibitem{barrett:p:2011}
{\sc J.~Barrett, H.~Garcke, and R.~N\"urnberg}, {\em The approximation of
  planar curve evolutions by stable fully implicit finite element schemes that
  equidistribute}, Num. Meth. Partial Diff. Eqs., 27 (2011), pp.~1--30.

\bibitem{bartels:p:2013}
{\sc S.~Bartels}, {\em A simple scheme for the approximation of the elastic
  flow of inextensible curves}, IMA J. Numer. Anal., 33 (2013), pp.~1115--1125.

\bibitem{baur:p:2013}
{\sc B.~Baur, M.~Grothaus, and T.~Thanh~Mai}, {\em Analytically weak solutions
  to linear {SPDE}s with unbounded time-dependent differential operators and an
  application}, Commun. Stoch. Anal., 7 (2013), pp.~551--571.

\bibitem{baus:p:2015}
{\sc F.~Baus, A.~Klar, N.~Marheineke, and R.~Wegener}, {\em
  Low-{M}ach-number--slenderness limit for elastic rods}, SIAM J. Appl. Math.,
  (2015, to appear).

\bibitem{berezansky:b:1995}
{\sc Y.~M. Berezansky and Y.~G. Kondratiev}, {\em Spectral Methods in
  Infinite-Dimensional Analysis}, vol.~12/1 of Mathematical Physics and Applied
  Mathematics, Kluwer Academic Publishers, Dordrecht, 1995.
\newblock Translated from the 1988 Russian original by P.V.~Malyshev and
  D.V.~Malyshev and revised by the authors.

\bibitem{bertails:p:2006}
{\sc F.~Bertails, B.~Audoly, M.~Cani, B.~Querleux, F.~Leroy, and
  J.~L\'ev\'eque}, {\em Super-helices for predicting the dynamics of natural
  hair}, ACM Transaction Graphics, 25 (2006), pp.~1180--1187.

\bibitem{bonilla:p:2007}
{\sc L.~L. Bonilla, T.~G{\"o}tz, A.~Klar, N.~Marheineke, and R.~Wegener}, {\em
  Hydrodynamic limit for the {F}okker-{P}lanck equation describing fiber
  lay-down models}, SIAM J. Appl. Math., 68 (2007), pp.~648--665.

\bibitem{brzezniak:p:2005}
{\sc Z.~Brze{\'z}niak, B.~Maslowski, and J.~Seidler}, {\em Stochastic nonlinear
  beam equations}, Probab. Theory Related Fields, 132 (2005), pp.~119--149.

\bibitem{coleman:p:1993}
{\sc B.~D. Coleman, E.~H. Dill, M.~Lembo, Z.~Lu, and I.~Tobias}, {\em On the
  dynamics of rods in the theory of {K}irchhoff and {C}lebsch}, Arch. Rat.
  Mech. Anal., 121 (1993), pp.~339--359.

\bibitem{deckelnick:p:2009}
{\sc K.~Deckelnick and G.~Dziuk}, {\em Error analysis for the elastic flow of
  parameterized curves}, Math. Comp., 78 (2009), pp.~645--671.

\bibitem{dziuk:p:2002}
{\sc G.~Dziuk, E.~Kuwert, and R.~Sch\"atzle}, {\em Evolution of elastic curves
  in $\mathbb{R}^n$: {E}xistence and computation}, SIAM J. Math. Anal., 33
  (2002), pp.~1228--1245.

\bibitem{grothaus:p:2008}
{\sc M.~Grothaus and A.~Klar}, {\em Ergodicity and rate of convergence for a
  non-sectorial fiber lay-down process}, SIAM J. Math. Anal., 40 (2008),
  pp.~968--983.

\bibitem{hinze:b:2009}
{\sc M.~Hinze, R.~Pinnau, M.~Ulbrich, and S.~Ulbrich}, {\em Optimization with
  {PDE} Constraints}, vol.~23 of Mathematical Modelling: {T}heory and
  Applications, Springer, New York, 2009.

\bibitem{juengel:p:2001}
{\sc A.~J\"ungel and R.~Pinnau}, {\em A positivity-preserving numerical scheme
  for a nonlinear fourth order parabolic system}, SIAM J. Num. Anal., 39
  (2001), pp.~385--406.

\bibitem{klar:p:2009}
{\sc A.~Klar, N.~Marheineke, and R.~Wegener}, {\em Hierarchy of mathematical
  models for production processes of technical textiles}, ZAMM - J. Appl. Math.
  Mech., 89 (2009), pp.~941--961.

\bibitem{landau:b:1970}
{\sc L.~D. Landau and E.~M. Lifschitz}, {\em Theory of Elasticity}, vol.~VII of
  A Course of Theoretical Physics, Pergamom Press, Oxford, 1970.

\bibitem{marheineke:p:2006}
{\sc N.~Marheineke and R.~Wegener}, {\em Fiber dynamics in turbulent flows:
  {G}eneral modeling framework}, SIAM J. Appl. Math., 66 (2006),
  pp.~1703--1726.

\bibitem{marheineke:p:2007}
\leavevmode\vrule height 2pt depth -1.6pt width 23pt, {\em Fiber dynamics in
  turbulent flows: {S}pecific {T}aylor drag}, SIAM J. Appl. Math., 68 (2007),
  pp.~1--23.

\bibitem{marheineke:p:2011}
\leavevmode\vrule height 2pt depth -1.6pt width 23pt, {\em Modeling and
  application of a stochastic drag for fiber dynamics in turbulent flows}, Int.
  J. Multiphase Flow, 37 (2011), pp.~136--148.

\bibitem{mesirov:b:1996}
{\sc J.~P. Mesirov, K.~Schulten, and D.~W. Sumners}, eds., {\em Mathematical
  Approaches to Biomolecular Structure and Dynamics}, Springer, New York, 1996.

\bibitem{mora:p:2003}
{\sc M.~G. Mora and S.~M\"uller}, {\em Derivation of the nonlinear
  bending-torsion theory for inextensible rods by {$\Gamma$}-convergence},
  Calc. Var. Partial. Diff. Eqs., 18 (2003), pp.~287--305.

\bibitem{oelz:p:2011}
{\sc D.~B. \"Olz}, {\em On the curve straightening flow of inextensible, open,
  planar curves}, S$\vec{e}$MA J., 54 (2011), pp.~5--24.

\bibitem{pearson:b:1985}
{\sc J.~R.~A. Pearson}, {\em Mechanics of Polymer Processing}, Elsevier, New
  York, 1985.

\bibitem{reed:b:1980}
{\sc M.~Reed and B.~Simon}, {\em Functional Analysis}, vol.~I of Methods of
  Modern Mathematical Physics, Academic Press, New York, 2~ed., 1980.

\bibitem{troeltzsch:b:2010}
{\sc F.~Tr\"oltzsch}, {\em Optimal Control of Partial Differential Equations --
  Theory, Methods and Applications}, vol.~112 of Graduate Studies in
  Mathematics, American Mathematical Society, Providence, Rhode Island, 2010.

\end{thebibliography}

\end{document}